\documentclass[11pt]{amsart}

\usepackage{tikz}
\usetikzlibrary{arrows}
\usetikzlibrary{cd,decorations.pathmorphing}

\usepackage{amsthm}
\usepackage{amssymb}
\usepackage{amsmath}
\usepackage[shortlabels]{enumitem}

\input xy
\xyoption{all} 

\newtheorem{theorem}{Theorem}
\newtheorem*{theorem*}{Theorem}

\newtheorem{example}[theorem]{Example}

\newtheorem{definition}[theorem]{Definition}
\newtheorem*{definition*}{Definition}
\newtheorem{lemma}[theorem]{Lemma}

\newtheorem{proposition}[theorem]{Proposition}
\newtheorem{remark}[theorem]{Remark}
\newtheorem{cor}[theorem]{Corollary}
\newtheorem*{cor*}{Corollary}

\newtheorem*{conjecture*}{Conjecture}
\numberwithin{theorem}{section}

\renewcommand\iff{%
\ifmmode\text{ if and only if }%
\else if and only if \fi}

\renewcommand{\and}{\wedge}

\renewcommand{\phi}{\varphi}

\newcommand{\Span}{\text{Span}}

\newcommand{\udim}{\underline{\text{dim}}}

\newcommand{\Mod}{\textnormal{Mod}}
\renewcommand{\mod}{\text{mod}}
\newcommand{\ind}{\text{ind}}

\newcommand{\Ab}{\text{Ab}}

\newcommand{\rad}{\textnormal{rad}}

\newcommand{\D}{\textnormal{D}}

\newcommand{\End}{\textnormal{End}}
\newcommand{\Hom}{\textnormal{Hom}}
\newcommand{\Ext}{\textnormal{Ext}}

\newcommand{\zg}{\textnormal{Zg}}
\newcommand{\Zg}{\textnormal{Zg}}

\newcommand{\pinj}{\textnormal{pinj}}

\newcommand{\mcal}[1]{\mathcal{#1}}

\newcommand{\st}{\ \vert \ }
\newcommand{\vertl}{\left\vert}
\newcommand{\vertr}{\right\vert}

\newcommand{\pp}{\textnormal{pp}}

\newcommand{\R}{\mathbb{R}}

\newcommand{\Q}{\mathbb{Q}}
\newcommand{\N}{\mathbb{N}}
\newcommand{\Z}{\mathbb{Z}}
\newcommand{\K}{\mathbb{K}}

\newcommand{\coker}{\text{coker}}

\newcommand{\X}{\mathbb{X}}

\newcommand{\coh}{\text{coh}}

\title[Decidability of modules over tubular algebras]{Decidability of theories of modules over tubular algebras}
\author{Lorna Gregory}
\thanks{The author acknowledges the support of EPSRC through Grant EP/K022490/1.}
\address{Dipartimento di Matematica e Fisica, Universit\`a degli Studi della Campania ``Luigi
Vanvitelli'', Viale Lincoln 5, 81100 Caserta, Italy}
\email{Lorna.Gregory@gmail.com}
\date{\today}

\newcommand\hfuzzReset{\hfuzz=3pt}
\hfuzzReset
\newcommand\toleranceReset{\tolerance=1400}
\toleranceReset
\newcommand\emergencystretchReset{\emergencystretch=1ex}
\emergencystretchReset

\begin{document}
\maketitle

\begin{abstract}
We show that the common theory of all modules over a tubular algebra (over a recursive algebraically closed field) is decidable. This result supports a long standing conjecture of Mike Prest which says that a finite-dimensional algebra (over a recursively given field) is tame if and only its common theory of modules is decidable \cite{PreBk}. Moreover, as a corollary, we are able to confirm this conjecture for the class of concealed canonical algebras over algebraically closed fields. These are the first examples of non-domestic algebras which have been shown to have decidable theory of modules. We also correct results in \cite{modirrslope}, in particular, corollary 8.8 of that paper.
\end{abstract}

\tableofcontents
\section{Introduction}

The study of the decidability and undecidability of theories of modules over finite-dimensional algebras began with papers of Baur which showed that the $4$-subspace problem is decidable \cite{Baurfourss} and that the $5$-subspace problem is undecidable \cite{decundectheoriesabgroups} (also see \cite{SlobFrid}).
For a given ring $R$, the theory of $R$-modules is said to be decidable if there is an algorithm that decides whether a given first order sentence in the language of $R$-modules is true in all $R$-modules.
It follows easily from the results of Baur that the theory of modules over the path algebras of $\widetilde{D_4}$ (in subspace orientation) is decidable and that the theory of modules over the path algebra of $\widetilde{D_5}$ (in subspace orientation) is undecidable.

Geisler \cite{Geislerthesis} and Prest \cite{decmiketameher} showed that the theory of modules over all tame hereditary algebras (over recursive fields with splitting algorithms) is decidable. In the converse direction, Prest showed that the theory of modules of strictly wild algebras is undecidable \cite{Epiintandstrwild} and thus all wild finite-dimensional hereditary algebras have undecidable theories of modules. Improving this result, the author, together with Prest, has shown that, over an algebraically closed field, all (finitely) controlled wild algebras have undecidable theories of modules \cite{intandreptype}. Note that at this time there is no known finite-dimensional algebra (over an algebraically closed field) which is wild but not (finitely) controlled wild. Indeed, Ringel has conjectured that all wild algebras (over algebraically closed fields) are controlled wild.

Toffalori and Puninski \cite{decfinrings} have worked on the problem of classifying finite commutative rings which have decidable theories of modules, which of course includes all commutative finite-dimensional algebras over finite fields.

\

The main result of this paper is the following.

\begin{theorem*}
Let $R$ be a tubular algebra over a recursive algebraically closed field. The common theory of $R$-modules is decidable.
\end{theorem*}

Our result supports the following long standing conjecture of Mike Prest.

\begin{conjecture*}\cite{PreBk}
Let $R$ be a finite-dimensional algebra over a suitably recursive field. The theory of $R$-modules is decidable if and only if $R$ is of tame representation type.
\end{conjecture*}

Tubular algebras are finite-dimensional non-domestic tame algebras of linear growth (see \cite[3.6]{Skopolgrowth} where tubular algebras are referred to as Ringel algebras). These are the first examples of non-domestic algebras which have been shown to have decidable theory of modules.

A finite-dimensional $k$-algebra is of tame representation type if for each dimension $d$, almost all $d$-dimensional modules are in $\mu(d)\in\N$ many $1$-parameter families (for a precise definition see \cite[3.3]{ss3}). An algebra is of domestic representation type if there is a finite bound on $\mu(d)$. So, the module categories of non-domestic algebras are significantly more complex than those of domestic representation type. A finite-dimensional $k$-algebra $R$ is of wild representation type if there is an exact $k$-linear functor $F:\text{fin}\text{-}k\langle x,y\rangle\rightarrow \mod\text{-}R$ which preserves indecomposability and reflects isomorphism type, where $\text{fin}\text{-}k\langle x,y\rangle$ denotes that category of finite-dimensional right modules over the free $k$-algebra in two non-commuting variables. Drozd, \cite{Drozd}, showed that all finite-dimensional algebras over algebraically closed fields are either wild or tame and not both.

\

Tubular algebras, introduced by Ringel in \cite{Ringeltub}, belong to a wider class of algebras called the concealed canonical algebras. According to \cite{LenPensep}, the concealed canonical algebras are exactly those algebra which admit a sincere separating tubular family of stable tubes. Equivalently, they are exactly the endomorphism rings of tilting bundles in categories of coherent sheaves on Geigle-Lenzing weighted projective lines. Moreover, \cite[3.6]{LenMeltilt}, the tubular algebras are exactly the tame non-domestic concealed canonical algebras; this perspective gives a geometric interpretation of the category of finite-dimensional modules over a tubular algebra akin to Atiyah's description of the category for coherent sheaves on an elliptic curve. As a corollary to our main theorem, we are able to conclude, see \ref{concandec}, that Prest's conjecture is true for all concealed canonical algebras.

Our methods for proving our main result are inspired by results of Harland and Prest in \cite{modirrslope}, an understanding of the Ziegler topology for modules of a fixed rational slope, decidability for tame hereditary algebras and decidability of Presburger arithmetic. For most of the paper we will work with general tubular algebras. However, in section \ref{1pointext}, we will mainly deal with canonical algebras of tubular type.

We also show that corollary 8.8 of \cite{modirrslope} is false and provide, see \ref{replcor}, a best possible replacement for that result.

\section{Background}\label{background}
If $R$ is a ring then we write $\mod\text{-}R$ for the category of finitely presented right $R$-modules, $\Mod\text{-}R$ for the category of all right $R$-modules and $\ind\text{-}R$ for the set of isomorphism classes of finitely presented indecomposable right $R$-modules. If $M,N\in \Mod\text{-}R$ then we will frequently write $(N,M)$ for $\Hom_R(N,M)$.

If $\mcal{X}$ is a class of modules then we write $(M,\mcal{X})=0$ (respectively $(\mcal{X},M)=0$) to mean that $(M,X)=0$ (respectively $(X,M)=0$) for all $X\in \mcal{X}$.

We will usually assume that finite-dimensional algebras are basic, connected and over an algebraically closed base field. Note, however, that every finite-dimensional algebra is ($k$-linearly) Morita equivalent to a basic algebra and every basic finite-dimensional algebra is isomorphic to a finite product of basic connected algebras. So, for the main results of this article, restricting to basic connected finite-dimensional algebras is only a cosmetic restriction.

\subsection{Grothendieck groups and the Euler characteristic}

Let $R$ be a finite-dimensional algebra, let $S_1,\ldots,S_n$ be the simple modules over $R$ and for $1\leq i\leq n$, let $P_i$ be the projective cover of $S_i$. If $M$ is a finite-dimensional module over $R$ then we call \[\udim M=(\dim\Hom_R(P_1,M),\ldots,\dim\Hom_R(P_n,M)),\] the \textbf{dimension vector} of $M$.  The \textbf{Grothendieck group}, $K_0(R)$, of a finite-dimensional algebra $R$ is the abelian group generated by the isomorphism classes $[X]$ of modules $X\in\mod\text{-}R$ and subject to the relation $[Y]-[X]-[Z]=0$ whenever there exists a short exact sequence
\[0\rightarrow X\rightarrow Y\rightarrow Z\rightarrow 0.\]

Note that $K_0(R)\cong \Z^n$. Moreover, we identify $K_0(R)$ with $\Z^n$ via the unique isomorphism which for all $M\in\mod\text{-}R$ sends $[M]$ to $\udim M$.

We say a vector $x=(x_1,\ldots,x_n)\in K_0(R)$ is \textbf{positive} if $x_i\geq 0$ for $1\leq i\leq n$ and $x\neq 0$. Note that $x$ is positive if and only if $x=\udim M$ for some non-zero $M\in\mod\text{-}R$.

The assumption that $R$ is basic implies, \cite[II.2]{Ass1}, that there exists a quiver $Q$, with vertices corresponding to the simple $R$-modules, such that $R$ is isomorphic to the quotient of the path algebra $kQ$ by an admissible ideal. We say $x\in K_0(R)$ is \textbf{connected} if its support is connected in the underlying quiver $Q$ of $R$.

The Grothendieck group of a finite-dimensional algebra $R$ of finite global dimension can be equipped with a bilinear form $\langle -,- \rangle$, called the \textbf{Euler characteristic}, such that for all $M,N\in\mod\text{-}R$, \[\langle [M],[N] \rangle:=\sum_{i=0}^\infty(-1)^i\dim\Ext^i(M,N),\] see \cite[III.3.13]{Ass1}.


The \textbf{Euler quadratic form} of $R$ is defined as $\chi_R(x):=\langle x,x\rangle$. We call an element $x\in K_0(R)$ \textbf{radical} if $\chi_R(x)=0$ and a \textbf{root} if $\chi_R(x)=1$. We denote the set of radical vectors $\rad\chi_R$.

\subsection{Tubular algebras}
We will not give the definition of a tubular algebra in terms of branch extensions of tame concealed algebras, for this see \cite[Chapter 5]{Ringeltub} or \cite[XIX 3.19]{ss3}, instead we will describe their module categories.

As mentioned in the introduction, another route to tubular algebras is via coherent sheaves on Geigle-Lenzing weighted projective lines. We will use this perspective in section \ref{corharlandprest} and briefly in section \ref{maintheorem}. Introductory material and references on this topic are contained in \ref{cohsheaves}.

The Euler quadratic form of a tubular algebra is positive semi-definite. It follows from \cite[1.1.1]{Ringeltub}, that $x\in\rad\chi_R$ if and only if $\langle x,y\rangle+\langle y,x\rangle=0$ for all $y\in K_0(R)$. So, in particular,
\begin{enumerate}
\item $\rad\chi_R$ is a subgroup of $K_0(R)$,
\item if $x,y\in \rad\chi_R$ then $\langle x,y\rangle=-\langle y,x\rangle$ and
\item if $x\in\rad \chi_R$ and $y\in K_0(R)$ then $\chi_R(x+y)=\chi_R(y)$.
\end{enumerate}

If $R$ is a tubular algebra then there exists a canonical pair of linearly independent radical vectors $h_0,h_\infty$ which generate a subgroup of $\rad\chi_R$ of finite index \cite[Section 5.1]{Ringeltub}.

For finite-dimensional indecomposable modules $M$ over $R$ we define the \textbf{slope} of $M$ to be
\[\text{slope}(M)= -\frac{\langle h_0,\udim M\rangle}{\langle h_\infty, \udim M\rangle}.\]

For $q\in\mathbb{Q}^\infty_0\footnote{By $\mathbb{Q}^\infty_0$ we mean the non-negative rational together with a maximal element $\infty$}$, let $\mcal{T}_q$ be the set of isomorphism classes of indecomposable finite-dimensional modules of slope $q$. Let $\mcal{P}_0$ be the preprojective component (the indecomposable finite-dimensional modules $M$ with $\langle h_0,\udim M \rangle<0$ and $\langle h_\infty,\udim M\rangle\leq 0$) and let $\mcal{Q}_\infty$ be the preinjective component (the indecomposable finite-dimensional modules $M$ with $\langle h_0,\udim M \rangle\geq 0$ and $\langle h_\infty,\udim M\rangle> 0$).

\begin{theorem}\cite{Ringeltub}
Let $R$ be a tubular algebra. The set of indecomposable finite-dimensional modules is
\[\mcal{P}_0\cup\bigcup\{\mcal{T}_q\ \st q\in \Q^\infty_0\}\cup\mcal{Q}_\infty\] where each $\mcal{T}_q$ is a tubular family separating $\mcal{P}_q:=\mcal{P}_0\cup\bigcup\{\mcal{T}_{q'}\ \st q'<q\}$ from $\mcal{Q}_q:=\bigcup\{\mcal{T}_{q'}\ \st q'>q\}\cup\mcal{Q}_\infty$.
\end{theorem}

That $\mcal{T}_q$ separates $\mcal{P}_q$ from $\mcal{Q}_q$ means that $0=(\mcal{T}_q,\mcal{P}_q)=(\mcal{Q}_q,\mcal{T}_q)=(\mcal{Q}_q,\mcal{P}_q)$ and, that for every tube $\mcal{T}(\rho)\in \mcal{T}_q$ and every homomorphism from $L\in\mcal{P}_q$ to $M\in\mcal{Q}_q$, factors through a direct sum of modules in $\mcal{T}(\rho)$.

All finite-dimensional modules over tubular algebras have injective dimension $\leq 2$ and projective dimension $\leq 2$. This means that $\dim\Ext^n(-,-)$ terms of the Euler characteristic are zero for $n>2$. All finite-dimensional indecomposable modules of strictly positive rational slope have projective and injective dimensions less than or equal to $1$ \cite[3.1.5]{Ringeltub}. Thus for those modules
\[\langle -,- \rangle:=\dim\Hom(-,-)-\dim\Ext(-,-).\]


\begin{theorem}\cite[5.2.6, pg 278]{Ringeltub}\label{Ringelcontrol}
Let $R$ be a tubular algebra.
\begin{enumerate}
\item For any indecomposable finite-dimensional $R$-module $X$, $\udim X$ is either a connected positive root or a connected positive radical vector of $\chi_R$.
\item For any positive connected root vector $x\in K_0(R)$, there is a unique indecomposable module $X\in\mod\text{-}R$ with $\udim X=x$.
\item For any positive connected radical vector $x\in K_0(R)$ there is an infinite family of indecomposable modules with $\udim X=x$.
\end{enumerate}
We will refer to the properties asserted in this theorem as $\mod\text{-}R$ is \textbf{controlled by} $\chi_R$.
\end{theorem}

If $R$ is a finite-dimensional $k$-algebra and $M$ is a right (respectively left) $R$-module then we write $M^*$ for its $k$-dual the left (respectively right) $R$-module $\Hom(M,k)$.

\begin{remark}\label{kdualslope}
Let $R$ be a tubular algebra. If $M$ is a finite-dimensional indecomposable with slope $a$ then the slope of $M^*$ is $1/a$, here we read $1/0$ as $\infty$ and $1/\infty$ as $0$.
If $M$ is preprojective (respectively preinjective) then $M^*$ is preinjective (respectively preprojective).
\end{remark}

The definition of slope on finite-dimensional indecomposable modules is extended to infinite-dimensional modules by Reiten and Ringel in \cite{InfdimcanalgReitenRingel} as follows.

\begin{definition}
Let $r\in \R_0^\infty$\footnote{By $\R_0^\infty$ we mean the non-negative reals together with a maximal element $\infty$}. We say a module $M$ is of \textbf{slope} $r$ if $(M,P)=0$ and $(Q,M)=0$ for all $P\in\mcal{P}_r$ and $Q\in\mcal{Q}_r$. This is equivalent to saying that $M\otimes P^*=0$ and $(Q,M)=0$ for all $P\in\mcal{P}_r$ and $Q\in\mcal{Q}_r$.
\end{definition}

\begin{theorem}\cite[13.1]{InfdimcanalgReitenRingel}
All indecomposable modules, except for the finite-dimensional preprojectives and preinjectives, over a tubular algebra have a slope.
\end{theorem}

\subsection{pp-formulas}

We now recall some concepts and results from model theory of modules; the necessary background can be found in \cite{PSL} or \cite{PreBk}.

A \textbf{pp-$n$-formula} is a formula in the language $\mcal{L}_R=(0,+,(\cdot r)_{r\in R})$ of (right) $R$-modules of the form
\[\exists \overline{y} (\overline{x},\overline{y})H=0\] where $\overline{x}$ is a $n$-tuple of variables and $H$ is an appropriately sized matrix with entries in $R$. If $\phi$ is a pp-formula and $M$ is a right $R$-module then $\phi(M)$ denotes the set of all elements $\overline{m}\in M^n$ such that $\phi(\overline{m})$ holds. Note that for any module $M$, $\phi(M)$ is a subgroup of $M^n$ equipped with the addition induced by addition in $M$. We identify two pp-formulas if they define the same subgroup in every $R$-module, equivalently in every finitely presented $R$-module \cite[1.2.23]{PSL}. After apply this identification, for each $n\in\N$ the set of (equivalence classes of) pp-$n$-formulas forms a lattice with the order given by implication, i.e. $\psi\leq \phi$ if and only if $\psi(M)\subseteq\phi(M)$ for all $R$-modules $M$. The meet of two pp-$n$-formulas $\phi,\psi$ is given by $\phi\wedge\psi$ and the join is given by $\phi+\psi$.

If $M$ is finitely presented module and $\overline{m}\in M^n$ then there is a pp-$n$-formula $\phi$ which \textbf{generates the pp-type} of $\overline{m}$ in $M$, that is, for all pp-formulas $\psi$, $\psi\geq\phi$ if and only if $\overline{m}\in\psi(M)$. Conversely, if $\phi$ is a pp-$n$-formula, then there exists a finitely presented module $M$ and $\overline{m}\in M^n$ such that $\phi$ generates the pp-type of $\overline{m}$ in $M$. We call $M$ together with $\overline{m}$ a \textbf{free-realisation} of $\phi$. For proofs of these assertions and more about free-realisations, see \cite[Section 1.2.2]{PSL}.

A pair of pp-formulas $\phi/\psi$\footnote{This notation for a pair of pp-formulae, which is standard in the area, is used to indicate that we are interested in the quotient groups $\phi(M)/\psi(M)$ for $R$-modules $M$.} is a \textbf{pp-$n$-pair} if for all $R$-modules $M$, $\phi(M)\supseteq \psi(M)$. We say that a pp-pair $\phi/\psi$ is \textbf{open} on $M$ if $\phi(M)\neq \psi(M)$ and \textbf{closed} on $M$ if $\phi(M)=\psi(M)$.

A functor from $\Mod\text{-}R$ to $\Ab$ is said to be \textbf{coherent} if it is additive and if it commutes with products and direct limits. Every pp-pair gives rise to a coherent (additive) functor $\phi/\psi:\Mod\text{-}R\rightarrow\Ab$ by sending $M\in\Mod\text{-}R$ to $\phi(M)/\psi(M)$. All coherent functors arise in this way. Moreover these are exactly the functors $F$ such that there exist $A,B\in\mod\text{-}R$ and $f:B\rightarrow A$ such that
\[(A,-)\xrightarrow{(f,-)}(B,-)\rightarrow F\rightarrow 0\] is exact. See \cite[Section 10]{PSL}.

For example, the functors $(M,-)$ and $-\otimes M$ are coherent when $M$ is finitely-presented and hence equivalent to functors defined by pp-pairs.

\subsection{Definable subcategories and Ziegler spectra}

A \textbf{definable subcategory} of $\Mod\text{-}R$ is a subcategory which is closed under pure-submodules, taking direct limits and products. Equivalently, see \cite[3.47]{PSL}, a full subcategory $\mcal{D}$ of $\Mod\text{-}R$ is a definable subcategory if there is a set of pp-pairs $\Omega$ such that $M\in \mcal{D}$ if and only if $\phi(M)=\psi(M)$ for all $\phi/\psi\in\Omega$. If $X\subseteq \Mod\text{-}R$ then we will write $\langle X\rangle$ for the smallest definable subcategory containing $X$.

Let $R$ be a tubular algebra and $r\in \R_0^\infty$. Since $-\otimes P^*$ is a coherent functor for each finite-dimensional $P$ and $(Q,-)$ is a coherent functor for each finite-dimensional $Q$, the class of all modules of slope $r$ is a definable subcategory which we denote $\mcal{D}_r$.

Let $R$ be a ring. An embedding of $R$-modules $f:M\rightarrow N$ is \textbf{pure} if for every pp-$1$-formula $\phi$, $f(\phi(M))=\phi(N)\cap f(M)$. A right $R$-module $M$ is \textbf{pure-injective} if it is injective over all pure-embeddings.

The  (right) \textbf{Ziegler spectrum} of a ring $R$ is a topological space with set of points, $\pinj_R$, the isomorphism classes of indecomposable pure-injectives and basis of open sets given by
\[\left(\phi/\psi\right):=\{N\in\pinj_R \st \phi(N)\neq \psi(N)\}\] where $\phi/\psi$ is a pp-pair.

The open sets $\left(\phi/\psi\right)$ are exactly the compact open sets of $\Zg_R$. Note that this means that $\Zg_R$ itself is compact. We should mention here that the open sets of the form $\left(\phi/\psi\right)$ where $\phi$ and $\psi$ are pp-$1$-formulas are also a basis for $\Zg_R$.

There is a correspondence between closed subsets of $\Zg_R$ and definable subcategories of $\Mod\text{-}R$ given by taking a closed subset $\mcal{C}$ to the smallest definable subcategory $\langle \mcal{C}\rangle$ containing $\mcal{C}$ and in the opposite direction taking a definable subcategory $\mcal{D}$ to $\mcal{D}\cap\pinj_R$.

The following is an explanation of the above correspondence. If $\mcal{D}_1$ and $\mcal{D}_2$ are definable subcategories of $\Mod\text{-}R$ then $\mcal{D}_1=\mcal{D}_2$ if and only if $\mcal{D}_1\cap\pinj_R=\mcal{D}_2\cap\pinj_R$ \cite[5.1.5]{PSL}. Thus, if $\mcal{D}$ is a definable subcategory then $\mcal{D}=\langle \mcal{D}\cap\pinj_R\rangle$. Conversely, if $\mcal{C}$ is a closed subset of $\Zg_R$ and $N\in \langle \mcal{C}\rangle \cap\pinj_R$ then for all pp-pairs $\phi/\psi$, $\phi(N)\neq\psi(N)$ implies $\phi(M)\neq \psi(M)$ for some $M\in \mcal{C}$. Since $\mcal{C}$ is closed and the basis of $\Zg_R$ is given by pp-pairs, $N\in\mcal{C}$. Thus $\mcal{C}=\langle\mcal{C}\rangle\cap\pinj_R$.

Let $R$ be a tubular algebra and $r\in \R_0^\infty$. We denote the set of all indecomposable pure-injectives of slope $r$ by $\mcal{C}_r$. So $\mcal{C}_r=\mcal{D}_r\cap\pinj_R$.

For each rational $q\in \Q^+$\footnote{By $\Q^+$ we mean the strictly positive rationals} the indecomposable pure-injective modules in $\mcal{D}_q$ have been completely described.

\begin{lemma}\label{pureinjatrationalslope}\cite[Lemma 50]{Richardthesis}
The following is a complete list of the indecomposable pure-injective modules in $\mcal{D}_q$:

\begin{enumerate}
\item The modules in $\mcal{T}_q$
\item A unique Pr\"ufer module $S[\infty]$ for each quasi-simple $S$ in $\mcal{T}_q$
\item A unique adic module $\widehat{S}$ for each quasi-simple $S$ in $\mcal{T}_q$
\item The unique generic module $\mcal{G}_q$
\end{enumerate}
\end{lemma}

In section \ref{Zieglerrationalslope} we will describe the topology on $\mcal{C}_q$ for $q\in \Q^+$.

For $r\in\R^+$\footnote{By $\R^+$ we mean the strictly positive reals} irrational this has already been done in \cite{modirrslope}.

\begin{theorem}\cite[Theorem 8.5]{modirrslope}\label{onepointuptoelemeqatirrational}
Let $R$ be a tubular algebra and $r\in\R^+$ irrational. The definable subcategory $\mcal{D}_r$ has no non-trivial proper definable subcategories.
\end{theorem}

Two points $x,y$ in a topological space are said to be \textbf{topologically indistinguishable} if for all open sets $\mcal{U}$, $x\in \mcal{U}$ if and only if $y\in \mcal{U}$. Since the closed subsets of $\Zg_R$ correspond to definable subcategories of $\Mod\text{-}R$ this means that, after identifying topologically indistinguishable points, there is exactly one point in $\mcal{C}_r$ for $r\in\R^+$ irrational.

In section \ref{1pointext} we will use Prest's elementary duality for pp-formulas and Herzog's elementary duality for the Ziegler spectrum in order to transfer results about $\mcal{P}_0\cup\mcal{C}_0$ to results about $\mcal{C}_\infty\cup\mcal{Q}_\infty$.

A duality between the lattice of right pp-$n$-formulae and the lattice of left pp-$n$-formulae was first introduced by Prest \cite[Section 8.4]{PreBk} and then extended by Herzog \cite{herzogduality} to give an isomorphism between the lattice of open set of the left Ziegler spectrum of a ring and the lattice of open sets of the right Ziegler spectrum of a ring.

\begin{definition}
Let $\phi$ be a pp-$n$-formula in the language of right $R$-modules of the form $\exists \bar{y}(\bar{x},\bar{y})H=0$ where $\bar{x}$ is a tuple of $n$ variable, $\bar{y}$ is a tuple of $l$ variables, $H=(H' \ H'')^T$ and $H'$ (respectively $H''$) is a $n\times m$ (respectively $l\times m$) matrix with entries in $R$. Then $\D\phi$ is the pp-$n$-formula in the language of left $R$-modules $\exists \bar{z}\left(
                                                       \begin{array}{cc}
                                                         I & H' \\
                                                         0 & H'' \\
                                                       \end{array}
                                                     \right)\left(
                                                              \begin{array}{c}
                                                                \bar{x} \\
                                                                \bar{z} \\
                                                              \end{array}
                                                            \right)=0
$.

Similarly, let $\phi$ be a pp-$n$-formula in the language of left $R$-modules of the form $\exists \bar{y} \ H\left(
                                                                                                             \begin{array}{c}
                                                                                                               \bar{x} \\
                                                                                                               \bar{y} \\
                                                                                                             \end{array}
                                                                                                           \right)
=0$ where $\bar{x}$ is a tuple of $n$ variable, $\bar{y}$ is a tuple of $l$ variables, $H=(H' \ H'')$ and $H'$ (respectively $H''$) is a $m\times n$ (respectively $m\times l$) matrix with entries in $R$. Then $\D\phi$ is the pp-$n$-formula in the language of right $R$-modules $\exists \bar{z} \ (\bar{x},\bar{z})\left(
                                                       \begin{array}{cc}
                                                         I & 0 \\
                                                         H' & H'' \\
                                                       \end{array}
                                                     \right)=0
$.
\end{definition}

Note that the pp-formula $a|x$ for $a\in R$ is mapped by $D$ to a formula equivalent to $xa=0$ and the pp-formula $xa=0$ for $a\in R$ is mapped by $D$ to a formula equivalent to $a|x$.

\begin{theorem}\cite[Chapter 8]{PreBk}
The map $\phi\rightarrow D\phi$ induces an anti-isomorphism between the lattice of right pp-$n$-formulae and the lattice of left pp-$n$-formulae. In particular, if $\phi,\psi$ are pp-$n$-formulae then $D(\phi+\psi)$ is equivalent to $D\phi\wedge D\psi$ and $D(\phi\wedge\psi)$ is equivalent to $D\phi+D\psi$.
\end{theorem}

This gives rise ``at the level of open sets'' to a homeomorphism from the left Ziegler spectrum of $R$ to the right Ziegler spectrum of $R$. To be precise:

\begin{theorem}\cite{herzogduality}\label{lattice iso ziegler}
The map $D$ given on basic opens sets by
\[\left(\phi/\psi\right)\mapsto\left(D\psi/D\phi\right)\] is an idempotent lattice isomorphism from the lattice of open sets of $\zg_R$ to the lattice of open sets of $_R\zg$.
\end{theorem}

It is unknown whether this lattice isomorphism always comes from a homeomorphism or even if this map always comes from a homeomorphism between $\zg_R$ and $_R\zg$ after identifying topologically indistinguishable points in both spaces.

The lattice isomorphism between open sets in $\Zg_R$ and open sets in ${_R}\Zg$ gives rise to a lattice isomorphism between the lattices of closed sets.

\begin{remark}\label{dualityC0Cinf}
Under the lattice isomorphism $D$, $\mcal{C}_0$ in $\Zg_R$ is sent to $\mcal{C}_\infty$ in ${_R}\Zg$. This follows from the proofs of \ref{effrepfun} and \ref{efftenfun} in section \ref{Basiccals} and \ref{kdualslope}.
\end{remark}

\begin{lemma}\cite[1.3.13]{PSL}\label{kdualandelemdual}
Let $R$ be a finite-dimensional $k$-algebra. Let $\phi/\psi$ be a right pp-pair and $M$ a right $R$-module. Then $\psi(M)\leq\phi(M)$ if and only if $D\phi(M^*)\leq D\psi(M^*)$.
\end{lemma}

\subsection{Baur-Monk and decidability}

Let $\phi/\psi$ be a pp-$n$-pair and $n\in\N$. There is a sentence, denoted by $\vertl\phi/\psi\vertr\geq n$ in the language of (right) $R$-modules, which expresses in every $R$-module $M$ that the quotient group $\phi(M)/\psi(M)$ has at least $n$ elements. Such sentences will be referred to as \textbf{invariant sentences}.

\begin{theorem}[Baur-Monk Theorem]\cite{PreBk}
Let $R$ be a ring. Every sentence $\chi\in\mcal{L}_R$ is equivalent to a boolean combination of invariant sentences.
\end{theorem}

If $R$ is an algebra over an infinite field then for all pp-pairs $\phi/\psi$ and all $R$-modules $M$, $|\phi(M)/\psi(M)|$ is either equal to one or infinite. This is because if $M$ is a module over a $k$-algebra then $\phi(M)$ and $\psi(M)$ are $k$-vector subspaces of $M^n$ and thus so is $\phi(M)/\psi(M)$.

A \textbf{recursive field} is a field $k$ together with a bijection with $\N$ such that addition and multiplication in the field induce recursive functions on $\N$ via this bijection. If $k$ is a countable algebraically closed field then there exists a bijection $f:k\rightarrow \N$ so that $k$ together with $f$ is a recursive field. With a bit of work, this follows from \cite[5.1]{HandRecMath} together with the fact, \cite[2.2.9]{Markermodeltheory}, that the theory of algebraically closed fields of a specified characteristic is decidable.

We will frequently use the word ``effectively'' followed by an operation in this paper. For example, ``effectively calculate'', ``effectively decide'' or as in the next paragraph ``effectively list''. This is just short hand for there exists an algorithm which performs that operation.

If $R$ is a finite-dimensional algebra over a recursive field then the theory of $R$-modules is recursively axiomatisable i.e. we can effectively list axioms for the theory of $R$-modules. In this situation we may use the so called \emph{proof algorithm},
which lists all sentences that are true in all $R$-modules by listing all formal proofs
in first order logic from the axioms for the theory of $R$-modules.\footnote{The existence of such an algorithm may be found in any
standard source on first order logic, e.g. \cite{Enderton}}.

With the proof algorithm in hand, we may then compute, for each sentence $\Theta\in\mcal{L}_R$,
a Boolean combination $\chi$ of invariant sentences that is equivalent to $\Omega$ as follows:
In the list of formal proofs we search for entries of the form
 $\Omega\leftrightarrow \chi$ for some Boolean combination of invariant sentences $\chi$.
By Baur-Monk the search terminates.


Thus, given a finite-dimensional algebra $R$ over an algebraically closed recursive field $k$, in order to show that the theory of $R$-modules is decidable it is enough to show that there is an algorithm which, given a boolean combination $\chi$ of invariant sentences of the form $\vertl\phi/\psi\vertr>1$, answers whether there is an $R$-module in which $\chi$ is true.

If $\chi$ is a boolean combination of invariants sentences, we can put it into disjunctive normal form $\bigvee_{i=1}^n\chi_i$ where each $\chi_i$ is a conjunction of invariants sentences and negations of invariants sentences. It is therefore enough to be able to check whether one of the $\chi_i$ is true in some $R$-module.

Suppose $\chi$ is of the form
\[\bigwedge_{i=1}^n|\phi_i/\psi_i|>1\wedge\bigwedge_{j=1}^m|\sigma_j/\tau_j|=1\]
where $\phi_i/\psi_i$ and $\sigma_j/\tau_j$ are pp-$1$-pairs. Since every module is elementary equivalent to a (possibly infinite) direct sum of indecomposable pure-in\-jective modules \cite[4.36]{PreBk} and solution sets of pp-formulas commute with direct sums, there is an $R$-module $M$ which satisfies $\chi$ if and only if there are indecomposable pure-injective $R$-modules $M_1,\ldots,M_n$ such that $M_i$ satisfies
\[|\phi_i/\psi_i|>1\wedge\bigwedge_{j=1}^m|\sigma_j/\tau_j|=1\] for each $1\leq i\leq n$.

Thus, it is enough to show that there is an algorithm which given pp-$1$-pairs, $\phi/\psi,\phi_1/\psi_1,\ldots,\phi_n/\psi_n$  answers whether
\[\left(\phi/\psi\right)\subseteq\bigcup_{i=1}^n\left(\phi_i/\psi_i\right).\]


An \textbf{interpretation functor}, $I:\Mod\text{-}R\rightarrow \Mod\text{-}S$, is specified (up to equivalence) by giving a pp-$m$-pair $\phi/\psi$ and, for each $s\in S$, a pp-$2m$-formula $\rho_s$ such that, for all $M\in\Mod\text{-}R$, the solution set $\rho_s(M,M)\subseteq M^m\times M^m$ defines an endomorphism of $\phi(M)/\psi(M)$ as an abelian group, and such that $\phi(M)/\psi(M)$, together with the $\rho_s$ actions, is an $S$-module (see \cite{PreInterp} or \cite[18.2.1]{PSL}).

An interpretation functor $I:\Mod\text{-}R\rightarrow \Mod\text{-}S$ gives rise to a mapping $\chi\mapsto \chi'$ from the set of sentences in the language of $S$-modules to the set of sentences in the language of $R$-modules such that for any $R$-module $M$,  $\chi$ is true for $IM$ if and only if $\chi'$ is true for $M$. In particular, $\chi$ is true for all $S$-modules in the image of $I$ if and only if $\chi'$ is true for all $R$-modules.

If $R$ and $S$ are finite-dimensional $k$-algebras and $I:\Mod\text{-}R\rightarrow \Mod\text{-}S$ is a $k$-linear interpretation functor, then together with $\phi/\psi$, it is enough to specify pp-formulas $\rho_{s_1},\ldots,\rho_{s_n}$ where $s_1,\ldots,s_n$ are a $k$-basis for $S$ and then extend $k$-linearly. Moreover, if $k$ is a recursive field then the induced mapping on sentences in the previous paragraph is effective.

\section{Basic calculations}\label{Basiccals}

Let $k$ be a recursive field. In this section we list basic operations that can be carried out effectively over $k$ which we will need later in the paper. We will sketch proofs of some of the less trivial operations.

\begin{remark}
Given a finite subset $S$ of $k^n$ and $v\in k^n$ we can effectively calculate a basis for $\Span S$, decide whether $v\in \Span S$ and find a basis for a complement of $\Span S$. In particular, we can effectively calculate the dimension of $\Span S$.
\end{remark}

\begin{remark}
Given a matrix $M\in M_{n\times l}(k)$ we can effectively find a basis for the kernel of $M$ in $k^n$ and the image of $M$ in $k^l$ when considering $M$ as a linear map from $k^n$ to $k^l$. Hence we can calculate the rank of $M$.
\end{remark}

Let $R$ be a finite-dimensional algebra with $k$-basis $r_1=1,\ldots,r_s$. Let $\alpha_{ij}^k\in k$ be such that $r_ir_j=\sum_{k=1}^s\alpha_{ij}^kr_k$. These relations and that $r_1,\ldots,r_s$ is a $k$-basis for $R$ completely define $R$ as a $k$-algebra. An $R$-module $M$ is now given by a $k$-vector space $V$ together with linear maps $\phi_1,\ldots , \phi_s\in\Hom_k(V,V)$ such that $\phi_i\circ\phi_j=\sum_{k=1}^s\alpha_{ij}^k\phi_k$. Now, if $V$ is finite-dimensional, say of dimension $d$, by picking a basis $B$ for $V$, we may identify $M$ with $(k^d,A_1,\ldots,A_s)$ where $A_1,\ldots,A_s$ are $d\times d$ matrices with entries in $k$ representing the linear maps $\phi_1,\ldots,\phi_s$ with respect to $B$. We call $(k^d,A_1,\ldots,A_s)$ a \textbf{presentation} of $M$. Note that if $A_1,\ldots,A_s$ are $d\times d$ matrices with entries in $k$ then $(k^d,A_1,\ldots,A_s)$ is a presentation of an $R$-module if and only if $A_iA_j=\sum_{k=1}^s\alpha_{ij}^kA_k$ for $1\leq i,j\leq s$.

If $(k^d,A_1,\ldots,A_s)$ is a presentation of an $R$-module then we write \[M(A_1,\ldots,A_s)\] for the $R$-module it represents.

\begin{remark}\label{effhomspace}
Let $R$ be a finite-dimensional $k$-algebra with $k$-basis $r_1,\ldots,r_s$. Given presentations of $R$-modules $(k^n,A_1,\ldots, A_s)$ and $(k^l,B_1,\ldots,B_s)$ we can effectively calculate a basis for the subspace
\[M_{n\times l}(k)\supseteq\{\Phi\in M_{n\times l}(k) \st A_i\Phi=\Phi B_i\text{ for } 1\leq i\leq s\}.\] Note that this set is $\Hom(M(A_1,\ldots, A_s),M(B_1,\ldots,B_s))$ in terms of matrices with respect to the standard basis.
\end{remark}

\textbf{From now on we will assume that $k$ is a recursive algebraically closed field and that $R$ is a finite-dimensional algebra over $k$ given in terms of a $k$-basis $r_1,\ldots,r_s$ and relations.}

\begin{lemma}\label{effdecomp}
Given a presentation $(k^n,A_1,\ldots, A_s)$ of an $R$-module $M$, we can effectively decide whether $M$ is indecomposable or not. If $M$ is decomposable then we can effectively find non-zero pair-wise disjoint $A_1,\ldots, A_s$-invariant subspaces $V_1,\ldots,V_m$ of $k^n$ such that each $V_i$ with the restricted action of $A_1,\ldots A_s$ is indecomposable as an $R$-module and such that $V_1+\ldots+V_m=k^n$.
\end{lemma}
\begin{proof}
First we effectively find a basis for $\End_R(M)$. That is we find a basis $T_1,\ldots,T_l$ for the subspace
\[M_{n\times n}(k)\supseteq\{\Phi\in M_{n\times n}(k) \st A_i\Phi=\Phi A_i\text{ for } 1\leq i\leq s\}.\] We may assume $T_1$ is the identity matrix.

Now $M$ is indecomposable if and only if $\End_R(M)$ has no idempotents apart from $0$ and $1$. For $a_1,\ldots,a_l\in k$, the condition that $a_1T_1+\ldots+a_lT_l$ is idempotent is equivalent to $\overline{a}=(a_1,\ldots, a_l)$ being a root of a particular system of polynomial equations with coefficients in $k$ (and we can find this system effectively). Using effective quantifier elimination for algebraically closed fields, we can thus decide whether there exists $(a_1,\ldots,a_l)\in k$ such that $\overline{a}\neq 0$, $\overline{a}\neq (1,0,\ldots , 0)$ and $a_1T_1+\ldots+a_lT_l$ is idempotent. Thus, given a presentation of a finite-dimensional $R$-module, we can effectively decide if it is indecomposable or not.

Supposing that we know that $M$ is not indecomposable we may now search for an idempotent $e$ represented by $a_1T_1+\ldots +a_lT_l$ in $\End_R(M)$ which is not the identity or zero. We know we will eventually find one because $M$ is not indecomposable. Now $M=eM\oplus (e-1)M$ and we can easily use our presentation of $M$ to get presentations of $eM$ and $(e-1)M$. If either $eM$ or $(e-1)M$ is decomposable then we may repeat the process eventually stoping when we get a decomposition of $M$ into indecomposable summands.
\end{proof}

\begin{lemma}\label{Listingfd}
There is an algorithm which lists the indecomposable finite-dimensional representations of $R$.
\end{lemma}
\begin{proof}
%
%
%
%

Given \ref{effdecomp}, it is enough to be able to effectively decide if, given two presentations $(k^n,A_1,\ldots,A_s)$ and $(k^n,B_1,\ldots,B_s)$ of indecomposable $R$-modules $M$ and $N$, $M$ is isomorphic to $N$.

We can compute a basis $T_1,\ldots,T_l$ for
\[\Hom_A(M(A_1,\ldots,A_n), M(B_1,\ldots,B_n))\] and $S_1,\ldots,S_l$ for \[\Hom_A(M(B_1,\ldots,B_n), M(A_1,\ldots,A_n)).\] Now $M(A_1,\ldots,A_n)$ and $M(B_1,\ldots,B_n)$ are isomorphic if and only if there exist $t_1,\ldots,t_l\in k$ and $s_1,\ldots,s_l\in k$ such that $(t_1T_1+\ldots+t_lT_l)(s_1S_1+\ldots+s_lS_l)=1$. This can be expressed as a system of polynomial equations over $k$ in $t_1,\ldots,t_l,s_1,\ldots,s_l$ and thus we may check, using effective quantifier elimination for algebraically closed fields, whether it has a solution or not.
\end{proof}

\begin{lemma}\label{effppopen}
Given a presentation of a finitely presented $R$-module $M$ and a pp-pair $\phi/\psi$ we can effectively decide whether $\phi/\psi$ is open on $M$ or not.
\end{lemma}
\begin{proof}
Given a pp-$n$-formula $\phi$ we can calculate the dimension as a $k$-vector space of the solution set of $\phi$ in $M^n$ as follows. Suppose $\phi$ is
\[ \exists y_1,\ldots,y_m \bigwedge_{i=1}^l x_1r_{1i}+\ldots+x_nr_{ni}+y_1s_{1i}+\ldots+y_ms_{mi}=0 .\] The $k$-dimension of $\phi(M)$ is the $k$-dimension of the solution set of
\[\bigwedge_{i=1}^l x_1r_{1i}+\ldots+x_nr_{ni}+y_1s_{1i}+\ldots+y_ms_{mi}=0\] minus the $k$-dimension of the solution set of
\[\bigwedge_{i=1}^l y_1s_{1i}+\ldots+y_ms_{mi}=0 .\]

Since $\phi/\psi$ is a pp-pair, $\phi(M)\supseteq \psi(M)$, so the dimension of $\phi(M)/\psi(M)$ is $\dim_k\phi(M)-\dim_k\psi(M)$. So $\phi/\psi$ is open on $X$ if and only if $\dim_k\phi(M)>\dim_k\psi(M)$.
\end{proof}

\begin{lemma}\label{effppgen}
Given a presentation of a finite-dimensional module $M$ we can effectively find a pp-$n$-formula generating the pp-type of a generating tuple for $M$.
\end{lemma}
\begin{proof}
Let $(k^n,A_1,\ldots,A_s)$ be a presentation of $M$ and let $\overline{e}=(e_1,\ldots,e_n)$ be the $n$-tuple of standard basis vectors for $k^n$. We need to write down finitely many linear equations over $R$ which describe the linear relations over $R$ which $\overline{e}$ satisfies. We can do this by describing a system of finitely many linear equations over $k$ which describe the linear relations between the vectors $e_iA_j$ in $k^n$.
\end{proof}

\begin{lemma}\label{effrepfun}
Given a presentation of a finite-dimensional module $M$ we can effectively find a pp-$n$-pair $\phi/\overline{x}=0$ such that the functor $(M,-)$ is equivalent to the functor defined by $\phi/\overline{x}=0$.
\end{lemma}
\begin{proof}
Given a presentation $(k^n,A_1,\ldots,A_s)$ of $M$, by lemma \ref{effppgen}, we can effectively find a pp-$n$-formula generating the pp-type of the standard basis for $k^n$ in $M$.
\end{proof}

\begin{lemma}\label{efftenfun}
Given a presentation of a finite-dimensional module $M$ we can effectively find a pp-$n$-pair $\overline{x}=\overline{x}/\psi$ such that the functor $-\otimes M^*$ is equivalent to the functor defined by $\overline{x}=\overline{x}/\psi$.
\end{lemma}
\begin{proof}
If $\phi$ is a pp-formula then $D\phi$ denotes the elementary dual pp-formula in the sense of \cite[Section 1.3]{PSL}. Computing the elementary dual of a pp-formula is clearly effective.  If $\phi/\overline{x}=0$ is isomorphic to $\Hom(M^*,-)$ then $\overline{x}=\overline{x}/D\phi$ is isomorphic to $-\otimes M^*$, see \cite[Section 10.3]{PSL}. The previous lemma \ref{effrepfun} now finishes the proof.
\end{proof}

%

\begin{lemma}\label{effdimvector}
Given a presentation of a finitely presented $R$-module $M$ we can calculate its dimension vector. Hence, given a presentation of a finitely presented indecomposable $R$-module $M$ over a tubular algebra $R$, we can calculate the slope of $M$.
\end{lemma}
\begin{proof}
That we can calculate the dimension vector of a module now follows directly from \ref{effhomspace}.
\end{proof}

\begin{cor}\label{efflistslopeq}
We can list presentations of the finite-dimensional indecomposable modules of slope $q$. We can list the quasi-simples of slope $q$.
\end{cor}
\begin{proof}
The quasi-simples of slope $q$ are just those modules of slope $q$ with $1$-dimensional endomorphism ring.
\end{proof}

\begin{lemma}
Given a quasi-simple $S$ of slope $q$, we can list the finite-dimensional modules in the ray starting at $S$ and in the coray starting at $S$
\end{lemma}
\begin{proof}
Look for $M$ indecomposable of slope $q$ with $\Hom(S,M)\neq 0$ (respectively $\Hom(M,S)\neq 0$).
\end{proof}

\begin{lemma}\label{singlefdpoints}
There is an algorithm, which given a presentation of a finite-dimensional indecomposable $R$-module $M$, outputs a pp-$1$-pair $\phi/\psi$ isolating $M$ in $\Zg_R$.
\end{lemma}
\begin{proof}
Given a finite-dimensional indecomposable module $M$ over a finite-dimensional algebra, \cite{Conass} gives a method of effectively constructing an almost split sequence \[\xymatrix@C=0.5cm{
  0 \ar[r] & M \ar[rr]^{f} && L \ar[rr]^{g} && K \ar[r] & 0 }.\] Pick $m\in M$ non-zero and calculate $\phi$ generating the pp-type of $m$ and $\psi$ generating the pp-type of $f(m)$. Then $\phi/\psi$ isolates $M$ (see \cite[Theorem 5.3.31]{PSL}).

By \ref{effppgen} we can effectively find a pp-formula $\psi(\overline{x})$ generating the pp-type of the standard $k$-basis $(e_1,\ldots,e_n)$ of $M$. The pp-type of $m$ is $\exists \overline{y} \ (\phi(\overline{y})\wedge x=\sum_{i=1}^ny_ia_i)$ where $m=\sum_{i=1}^ne_ia_i$. Thus we can effectively find a pp-formula generating the pp-type of $m$.

\end{proof}

\begin{lemma}\label{excludingfinitelymanyfdpoints}
Given presentations of finitely many finite-dimensional indecomposable modules $X_1,\ldots,X_m$ and a pp-$1$-pair $\phi/\psi$ we can effectively find pp-pairs $\phi_1/\psi_1,\ldots,\phi_n/\psi_n$ such that $\left(\phi/\psi\right)\backslash\{X_1,\ldots,X_m\}=\bigcup_{i=1}^n\left(\phi_i/\psi_i\right)$.
\end{lemma}
\begin{proof}
Suppose we are given presentations of finitely many finite-dimensional indecomposable modules $X_1,\ldots,X_m$. For each $j$, let $\{X_j\}=(\sigma_j/\tau_j)$. By \ref{singlefdpoints}, we can find $\sigma_j/\tau_j$ effectively.  Since $\left(\phi/\psi\right)$ is compact and the set $\{X_1,\ldots,X_m\}$ is clopen, $\left(\phi/\psi\right)\backslash\{X_1,\ldots,X_m\}$ is compact. Thus there exists pp-pairs $\phi_1/\psi_1,\ldots,\phi_n/\psi_n$ such that \[\left(\phi/\psi\right)\backslash\{X_1,\ldots,X_m\}=\cup_{i=1}^n\left(\phi_i/\psi_i\right).\]

This is equivalent to  $\phi_1/\psi_1,\ldots,\phi_n/\psi_n$ being such that for all $1\leq i\leq n$, $X_j\notin \left(\phi_i/\psi_i\right)$ and
\[|\phi/\psi|>1\leftrightarrow\bigvee_{i=1}^n|\phi_i/\psi_i|>1\vee\bigvee_{j=1}^n|\sigma_j/\tau_j|>1.\]

We can now use the proof algorithm and \ref{effppopen} to search  for $\phi_1/\psi_1,\ldots,\phi_n/\psi_n$ such that for all $1\leq j\leq m$ and for all $1\leq i\leq n$, $X_j\notin \left(\phi_i/\psi_i\right)$ and
\[|\phi/\psi|>1\leftrightarrow\bigvee_{i=1}^n|\phi_i/\psi_i|>1\vee\bigvee_{j=1}^n|\sigma_j/\tau_j|>1.\]

\end{proof}

\begin{lemma}\label{recfreerealise}
There is an algorithm which given a pp-$1$-formula $\phi$ outputs a presentation of a finite-dimensional module $M$ and an element $m\in M$ such that $(M,m)$ freely realises $\phi$.
\end{lemma}

\begin{lemma}
There is an algorithm which, given a presentation of a module $M$ and an element $m$, outputs a presentation of $M/mR$.
\end{lemma}
\begin{proof}
We are using the fact that given a finite subset $L$ of $k^n$ we can algorithmically find a basis for $\Span L$ and a basis for a complement of $\Span L$.

Let $(k^n,A_1,\ldots,A_s)$ be a presentation for $M$ and identify $m$ with its image in $k^n$ with respect to this presentation.

Find a basis $e_1,\ldots,e_t$ for $\text{Span}\{m,mA_1,\ldots,mA_s\}$ and a basis for a complement $f_1,\ldots,f_l$ of $\text{Span}\{m,mA_1,\ldots,mA_s\}$. By considering the action of $A_1,\ldots,A_s$ on $f_1,\ldots,f_l$ we get a presentation of $M/mR$.
\end{proof}

\section{Ziegler Spectra of tubes of rational slope}\label{Zieglerrationalslope}

In this section we describe the Ziegler spectrum of $\mcal{D}_q$ where $q\in \Q^+$. That is we describe the induced topology on $\mcal{C}_q:=\Zg_R\cap \mcal{D}_q$ by describing the closed subsets of $\mcal{C}_q$.

Recall the complete list of indecomposable pure-injectives of rational slope $q$ from \ref{pureinjatrationalslope}.

We essentially follow Ringel's proof from \cite{ZgspcetameherRingel} for tame hereditary algebras.

\begin{proposition}\label{closedsubsets}
A subset $X$ of $\mcal{C}_q$ is closed if and only if the following hold:
\begin{enumerate}
\item If $S$ is a quasi-simple in $\mcal{T}_q$ and if there are infinitely many finite length modules $M\in X$ with $\Hom(S,M)\neq 0$ then $S[\infty]\in X$.
\item If $S$ is a quasi-simple in $\mcal{T}_q$ and if there are infinitely many finite length modules $M\in X$ with $\Hom(M,S)\neq 0$ then $\widehat{S}\in X$.
\item If there are infinitely many finite length modules in $X$ or $X$ contains an infinite length module then $\mcal{G}_q\in X$.
\end{enumerate}
\end{proposition}

\begin{lemma}
If $X$ is a closed subset of $\mcal{C}_q$ then (1) from \ref{closedsubsets} holds.
\end{lemma}
\begin{proof}
This is clear since the Pr\"ufer module is a direct union of such modules.
\end{proof}

\begin{lemma}
If $X$ is a closed subset of $\mcal{C}_q$ then (2) from \ref{closedsubsets} holds.
\end{lemma}
\begin{proof}
Suppose $X$ contains infinitely many finite-dimensional modules $M$ with $\Hom(M,S)\neq 0$. Then each of these $M$ is in the coray starting at $S$. Therefore $\widehat{S}$ is an inverse limit of these $M$. So by \cite[2.3]{ZieZardomstring}, $\widehat{S}$ is in the closure of these $M$.
\end{proof}

\begin{lemma}\label{closedzg3}
If $X$ is a closed subset of $\mcal{C}_q$ then (3) from \ref{closedsubsets} holds.
\end{lemma}
\begin{proof}
Since $X$ is closed it is compact. Therefore if $X$ contain infinitely many isolated points then $X$ must contain a non-isolated point. By \cite[5.3.31]{PSL}, all finite-dimensional points are isolated in the Ziegler spectrum of a finite-dimensional algebra and hence, all finite-dimensional points are isolated in $\mcal{C}_q$. Therefore, if $X$ contains infinitely many finite-dimensional indecomposable modules then $X$ must contain an infinite-dimensional module i.e. either an adic, Pr\"ufer or generic module. By \cite[8.10]{KrauseGeneric} we know that the generic is in the closure of every adic and every Pr\"ufer.
\end{proof}

\begin{proof}[Proof of \ref{closedsubsets}]
The proof of proposition \ref{closedsubsets} now is the same as Ringel's proof for tame hereditary algebras but working inside $\mcal{D}_q$.
\end{proof}

\begin{definition}
Let $E$ be a quasi-simple of slope $q$ and $i\in \N$. Let
\[R[E[i]]:=\{E[j] \st j\geq i\}\cup\{E[\infty]\}\] and
\[C[[i]E]:=\{[j]E \st j\geq i\}\cup\{\widehat{E}\}.\]
Note that, by \ref{closedsubsets},  both these sets are open in the subspace topology on $\mcal{C}_q$. We call open sets of the form $R[E[i]]$ rays and open sets of the form $C[[i]E]$ corays.
\end{definition}

We now classify the compact open subsets of $\mcal{C}_q$.

\begin{proposition}\label{compactopenatq}
The compact open subsets of $\mcal{C}_q$ are either cofinite (excluding only finite-dimensional points) or a finite union of rays and corays plus finitely many other finite-dimensional points.
\end{proposition}
\begin{proof}
If $\mcal{U}$ is an open set containing the generic then its complement only contains finitely many points, all of which are finite-dimensional by \ref{closedsubsets}. Such a set is compact because it is clopen i.e. also closed.

So we now consider compact open sets not containing the generic module. The set $\mcal{C}_q\backslash\{\mcal{G}_q\}$ is contained in a union of rays and corays. Thus, any compact open set not containing the generic is a subset of a finite union of rays and corays.  In particular, any compact open set not containing the generic contains only finitely many infinite-dimensional points.

If a compact open set only contains finite-dimensional points then it is finite (since these points are isolated). Suppose $\mcal{U}$ contains a Pr\"ufer point
$S[\infty]$. Then its complement, by \ref{closedsubsets}, must contain only finitely many points of the form $S[j]$. Thus for some $j$ the ray $R[S[j]]$ must be contained in $\mcal{U}$. Similarly, if $\mcal{U}$ contains an adic point
$\widehat{S}$ then it contains the coray $C[[j]S]$ for some $j$. Now removing all the rays and corays (which are open) we must be left with just finite-dimensional points. Since $\mcal{U}$ is compact, we are left with finitely many finite-dimensional points.
\end{proof}

\section{Algorithms at slope $q\in\Q^+$}\label{algslopeq}
In this section, we present an algorithm which, given $n+1$ pp-pairs \[\phi/\psi,\phi_1/\psi_1,\ldots,\phi_n/\psi_n\] and $q\in \Q^+$, answers whether
\[\mcal{C}_q\cap \left(\phi/\psi\right)\subseteq \bigcup_{i=1}^n\mcal{C}_q\cap \left(\phi_i/\psi_i\right).\]

Note that $\mcal{D}_q$ is axiomatised by saying that for each finite-dimensional indecomposable module $Q$ of slope strictly greater than $q$, the functor $(Q,-)$ is zero on $\mcal{D}_q$ and for each finite-dimensional indecomposable module $P$ of slope strictly less than $q$, the functor $-\otimes P^*$ is zero on $\mcal{D}_q$. Given a presentation of a module $M$, by lemma \ref{effrepfun} we can effectively find a pp-$n$-pair $\sigma/\tau$ such that the functor defined by $\sigma/\tau$ is equivalent to $(M,-)$ and by lemma \ref{efftenfun} a pp-$n$-pair $\sigma/\tau$ such that the functor defined by $\sigma/\tau$ is equivalent to $-\otimes M^*$. By \ref{Listingfd} and \ref{effdimvector} we can list the indecomposable finite-dimensional modules of slope $<q$ and those of slope $>q$. Thus, given $q\in\Q^+$, we can recursively list sentences which axiomatise $\mcal{D}_q$. Let $\Sigma_q$ be the recursive list of sentences axiomatising $\mcal{D}_q$.

\begin{remark}
Let $q\in \Q^+$ and $\phi/\psi,\phi_1/\psi_1,\ldots ,\phi_n/\psi_n$ be pp-pairs. Then \[\mcal{C}_q\cap\left(\phi/\psi\right)\subseteq \bigcup_{i=1}^n\mcal{C}_q\cap\left(\phi_i/\psi_i\right)\] if and only if
\[\Sigma_q\vdash|\phi/\psi|>1\rightarrow \bigvee_{i=1}^n|\phi_i/\psi_i|>1.\] By compactness, this means that there is some finite subset of $\Sigma\subseteq \Sigma_q$ such that
\[\Sigma \vdash |\phi/\psi|>1\rightarrow \bigvee_{i=1}^n|\phi_i/\psi_i|>1.\]
\end{remark}

We now use the results of the previous section to give canonical forms for compact open subsets of $\mcal{C}_q$.

\begin{lemma}\label{canformatq}
Each compact open subset $\mcal{U}$ of $\mcal{C}_q$ is unique of the form:
\begin{enumerate}
\item $F(\{X_1,\ldots,X_n\}):=\mcal{C}_q\backslash\{X_1,\ldots,X_n\}$ where $X_1,\ldots,X_n$ are finite-dimensional indecomposables of slope $q$
\item $\bigcup_{E\in S} R(E[j_E])\cup\bigcup_{E\in D} C([k_E]E)\cup\{X_1,\ldots,X_m\}$ where for each $E\in S$, if $R(E[i])\subseteq \mcal{U}$ then $i\geq j_E$, for each $E\in D$, if $C([i]E)\subseteq \mcal{U}$ then $i\geq k_E$ and for $1\leq i\leq m$, $X_i\notin \bigcup_{E\in S} R(E[j_E])\cup\bigcup_{E\in D} C([k_E]E)$.
\end{enumerate}

\end{lemma}
\begin{proof}
Proposition \ref{compactopenatq} gives us a description of the compact open subsets of $\mcal{C}_q$. We just need to observe that the list above contains no repeats.
\end{proof}

\begin{lemma}\label{algorforparticularq}
There is an algorithm, which given $q\in \Q^+$ and pp-pairs $\phi/\psi,\phi_1/\psi_1,\ldots,\phi_n/\psi_n$, answers whether
\[\mcal{C}_q\cap \left(\phi/\psi\right)\subseteq \bigcup_{i=1}^n\mcal{C}_q\cap \left(\phi_i/\psi_i\right).\]
\end{lemma}
\begin{proof}
By \ref{excludingfinitelymanyfdpoints}, \ref{singlefdpoints}, \ref{efftenfun} and \ref{effrepfun}, there is an algorithm which lists pp-pairs defining the open sets of the form $F(\{X_1,\ldots,X_n\})$, $\{X\}$, $C(X)$ and $R(X)$ where $X_1,\ldots,X_n,X\in\mcal{C}_q$.

Thus, since $\mcal{D}_q$ is recursively axiomatised, there is an algorithm which given a pp-pair $\phi/\psi$, finds a compact open set $\mcal{U}$ such that $(\phi/\psi)\cap\mcal{C}_q=\mcal{U}\cap\mcal{C}_q$ and such that $\mcal{U}$ is in the canonical form given in \ref{canformatq}.

We now need to take each of the compact open sets of the form\\
$F(\{X_1,\ldots,X_n\})$, $\{X\}$, $C(X)$ and $R(X)$  and write an algorithm which determines whether it is contained in a finite union, $W_1\cup\ldots\cup W_n$ of some specified other open sets of the form $F(\{X_1,\ldots,X_n\})$, $\{X\}$, $C(X)$ and $R(X)$.

\textbf{Case 0}: $\{X_1,\ldots,X_m\}\subseteq W_1\cup \ldots\cup W_n$

Check directly whether each $X_i$ is in some $W_j$. We can do this by \ref{effppopen}.

\textbf{Case 1}: $F(\{X_1,\ldots,X_m\})\subseteq W_1\cup \ldots\cup W_n$

If $F(\{X_1,\ldots,X_m\})\subseteq W_1\cup \ldots\cup W_n$ then, by \ref{compactopenatq}, one of the $W_i$s must be of the form $F(\{Y_1,\ldots,Y_l\})$ for some $Y_1,\ldots,Y_l$. So $F(\{Y_1,\ldots,Y_l\})$ contains all of $F(\{X_1,\ldots,X_m\})$ except the points $\{Y_1,\ldots,Y_l\}\backslash\{X_1,\ldots,X_m\}$. We now just need to check whether the finite subset $\{Y_1,\ldots,Y_l\}\backslash\{X_1,\ldots,X_m\}$ is contained in $W_1\cup\ldots\cup W_n$. This is case 0.

\textbf{Case 3}: $R(E[j])\subseteq W_1\cup \ldots\cup W_n$

If $R(E[j])\subseteq W_1\cup \ldots\cup W_n$ then $E[\infty]\in W_i$ for some $i$. So either one of the $W_i$ is of the form $F(\{X_1,\ldots,X_m\})$ for some $X_1,\ldots,X_m$ or is of the form $R[E[l]]$ for some $l$.

If one of the $W_i$s is $F(\{X_1,\ldots,X_m\})$ then check if any of $X_1,\ldots,X_m$ is of the form $E[k]$ for some $k\geq j$. For any which is, check if that point is in one of the remaining open sets.

If one of the $W_i$s is $R[E[l]]$ then either $l\leq j$, in which case $R(E[j])\subseteq R(E[l])$, or $l>j$. If $l>j$ then $R[E[j]]\backslash R[E[l]]$ is $E[j],E[j+1],\ldots,E[l-1]$. So we just use case 0 to find out whether these are contained in the remaining open sets.

\textbf{Case 4}:  $C([j]E)\subseteq W_1\cup \ldots\cup W_n$

As above but replacing $E[\infty]$ by $\widehat{E}$.
\end{proof}

\section{One-point extensions and coextensions}\label{1pointext}

Let $T_{n_1,\ldots,n_t}$ be the star quiver with $t$ arms of length $n_1,\ldots,n_t$ in the ``subspace'' configuration. A canonical algebra of tubular type is a one-point extension of the tame hereditary path algebra of $T_{n_1,\ldots,n_t}$ by a quasi-simple module $X$ at the base of a tube such that $(n_1,\ldots,n_t)$ is in the set $\{(3,3,3), (2,4,4), (2,3,6), (2,2,2,2)\}$ \cite[pg161]{Ringeltub}. These algebras may equally well be viewed as one-point coextensions of star path algebras with the ``cosubspace'' configuration by a quasi-simple module $X$ at the base of a tube.

Throughout this section $A$ will be the path algebra of a star quiver in subspace configuration as above, $X\in \mod\text{-}A$ will be a quasi-simple at the base of a tube and $A[X]$ will be the one-point extension of $A$ by $X$ i.e. the $2\times2$-matrix algebra
\[\left(
     \begin{array}{cc}
       A & 0 \\
       {_k}X_A & k \\
     \end{array}
   \right).
\]

The category $\Mod\text{-}A[X]$ is equivalent to $\text{Rep}(X)$, the category of representation of the bimodule ${_k}X_A$ and also to, $\overline{\text{Rep}}(X)$, the $k$-category with objects $M=(M_0,M_1,\Gamma_M)$ where $M_0$ is a $k$-vector space, $M_1$ is a right $A$-module and $\Gamma_M:M_0\rightarrow \Hom_A(X,M_1)$ is a $k$-vector space homomorphism. Morphisms in $\overline{\text{Rep}}(X)$ are given by pairs $f=(f_0,f_1):M=(M_0,M_1,\Gamma_M) \rightarrow N=(N_0,N_1,\Gamma_N)$, where $f_0:M_0\rightarrow N_0$ is a $k$-vector space homomorphism and $f_1:M_1\rightarrow N_1$ is a $A$-module homomorphism such that the following square commutes.

\[\xymatrix@R=30pt@C=60pt{
  M_0 \ar[d]_{\Gamma_M} \ar[r]^{f_0} & N_0 \ar[d]^{\Gamma_N} \\
  \Hom_A(X,M_1) \ar[r]^{\Hom_A(X,f_1)} & \Hom_A(X,N_1)   }\]
Throughout this section we identify representations of $A[X]$ with objects in $\overline{\text{Rep}}(X)$. So, if $(M_0,M_1,\Gamma_M)$ is a object in $\overline{\text{Rep}}(X)$, $m\in M_1$, $\delta\in M_0$, $a\in A$, $x\in X$ and $\mu\in k$ then

\[(m,\delta)\cdot \left(
                    \begin{array}{cc}
                      a & 0 \\
                      x & \mu \\
                    \end{array}
                  \right)=(m\cdot a+\Gamma_M(\delta)[x], \mu\delta).
\]

This gives us the following two embeddings of $\Mod\text{-}A$ into $\Mod\text{-}A[X]$; \[F_0:\Mod\text{-}A\rightarrow \Mod\text{-}A[X], \ \ M\mapsto (0,M,0)\] and \[F_1:\Mod\text{-}A\rightarrow \Mod\text{-}A[X], \ \  M\mapsto (\Hom(X,M),M,\text{Id}_{\Hom(X,M)}).\] We also have a functor $r:\Mod\text{-}A[X]\rightarrow \Mod\text{-}A$ which sends $(M_0,M_1,\Gamma_M)$ to $M_1$. This functor is right adjoint to $F_0$ and left adjoint to $F_1$.

In this section we will see that $F_0$ and $F_1$ are interpretation functors whose images are finitely axiomatisable and that every indecomposable pure-in\-jective in $\mcal{P}_0\cup\mcal{D}_0$ is in the  union of the images of $F_0$ and $F_1$.

\begin{remark}\label{F0intfunctorwithdefimage}
Let $A$ be a $k$-algebra and $X$ a right $A$-module. The assignment $F_0$ which sends a right $A$-module $M$ to the right $A[X]$-module $(0, M,0)$ and  sends a morphism $g:M\rightarrow N$ to $(0,g)$ is clearly a full and faithful exact interpretation functor whose image is a (finitely axiomatisable) definable subcategory of $\Mod\text{-}A[X]$.
\end{remark}

\begin{proposition}\label{F1intfunctorwithdefimage}
Let $A$ be a $k$-algebra and $X$ a finitely presented right $A$-module. The functor $F_1$ is a full and faithful (left exact) interpretation functor whose image is a definable subcategory (after closing under isomorphisms) of $\Mod\text{-}A[X]$.
\end{proposition}
\begin{proof}

It is straightforward to see that $F_1$ is indeed a functor and that it is full, faithful and left exact (see for instance \cite[1.4]{ss3}).

The functor $F_1$ is an interpretation functor if and only if it commutes with direct limits and products \cite[25.3]{defaddcats}. In order to check that $F_1$ commutes with direct limits and products, it is enough to check that its composition with the forgetful functor from $\Mod\text{-}A[X]$ to $\Mod\text{-}k$ commutes with direct limits and products. This follows since $\Hom(X,-)$ commutes with direct limits and products.

We now show that the image of $F_1$ is a (finitely axiomatisable) definable subcategory of $\Mod\text{-}A[X]$. First note that $L=(L_0,L_1,\Gamma_L)$ is in the (essential) image of $F_1$ if and only if $\Gamma_L$ is an isomorphism. Let $t_1,\ldots,t_n$ generate $X$ as an $A$-module. Note that for any $\delta\in L_0$ and $\gamma\in\Hom(X,L_1)$ we have that $\Gamma_L(\delta)=\gamma$ if and only if $\Gamma_L(\delta)[t_i]=\gamma[t_i]$ for $1\leq i\leq n$. Now for $\delta\in L_0$,

\[\Gamma_L(\delta)[t_i]=(0,\delta)\left(
                                    \begin{array}{cc}
                                      0 & 0 \\
                                      t_i & 0 \\
                                    \end{array}
                                  \right).\] Let $\psi\in \pp_R^n$ be the pp-formula
\[\exists z \bigwedge_{i=1}^nx_i=z\left(
                                    \begin{array}{cc}
                                      0 & 0 \\
                                      t_i & 0 \\
                                    \end{array}
                                  \right).
\]

Let $\phi$ generate the pp-type of $(t_1,\ldots,t_n)$ viewed as a tuple from $(0,X,0)$. Now
\[\phi(L)=\{f(\overline{t}) \st f\in\Hom_R((0,X,0),L)\}.\] So
\[\phi(L)=\{(\gamma[t_1],\ldots,\gamma[t_n]) \st \gamma\in \Hom_A(X,L_1)\}.\]

Thus $\Gamma_L$ is surjective if and only if $\phi(L)=\psi(L)$.


Let $e_0=\left(
           \begin{array}{cc}
             0 & 0 \\
             0 & 1 \\
           \end{array}
         \right)$.

Let \[\sigma(x):=\exists z \ x=ze_0\bigwedge_{i=1}^n x\left(
                                                               \begin{array}{cc}
                                                                 0 & 0 \\
                                                                 t_i & 0 \\
                                                               \end{array}
                                                             \right)=0
\]
For all $M\in\Mod\text{-}A[X]$, $\Gamma_M$ is injective if and only if $\sigma(M)=0$.


So $M$ is in the  (essential) image of $F_1$ if and only if $\sigma(M)=0$ and $\phi(M)=\psi(M)$.
\end{proof}

Let $\mcal{E}_0$ be the image of the functor $F_0$ and $\mcal{E}_1$ be the image of the functor $F_1$. We will show, \ref{fdinimF0F1}, that every indecomposable finite-dimensional module of slope $0$ is either contained in $\mcal{E}_0$ or $\mcal{E}_1$. Thus, if the finite-dimensional modules of slope $0$ are dense in the Ziegler closed subset corresponding to the definable subcategory of slope $0$ then all indecomposable pure-injective modules of slope zero are either contained in $\mcal{E}_0$ or $\mcal{E}_1$.

\begin{proposition}\label{fddenseinslopezero}
Let $R$ be a tubular algebra. The finite-dimensional indecomposable $R$-modules of slope zero are dense in the Ziegler closed subset of indecomposable pure-injective modules of slope zero.
\end{proposition}
\begin{proof}
By an argument exactly as the first paragraph of \cite[Theorem 13.6]{PreBk} we know that every open set containing a module of slope zero contains a finite-dimensional module of slope greater than or equal to zero.

Suppose $N$ is an infinite-dimensional indecomposable module of slope zero and that $F$ is a coherent functor with $FN\neq 0$. Let $P_2,P_1\in\mod\text{-}R$ be preprojective, $T_2,T_1\in\mod\text{-}R$ of slope zero, $Q_2,Q_1\in\mod\text{-}R$ of slope greater than zero and $f:P_1\oplus T_1\oplus Q_1\rightarrow P_2\oplus T_2\oplus Q_2$ be such that
\[(P_2\oplus T_2\oplus Q_2 ,-)\xrightarrow{(f,-)} (P_1\oplus T_1\oplus Q_1,-)\rightarrow F\rightarrow 0\] is exact.

Let $\pi_i:P_i\oplus T_i\oplus Q_i\rightarrow P_i\oplus T_i$ and $\mu_i:P_i\oplus T_i\rightarrow P_i\oplus T_i\oplus Q_i$ be canonical projections and embeddings for $i=1,2$. Let $G$ be a coherent functor such that
\[ (P_2\oplus T_2,-)\xrightarrow{(\pi_2\circ f\circ \mu_1,-)} (P_1\oplus T_1,-) \rightarrow G \rightarrow 0\] is exact.

Suppose that $M$ is an indecomposable module of slope zero. We show that $(\pi_2f\mu_1,M)$ is surjective if and only if $(f,M)$ is surjective. That is, we show that $FM\neq 0$ if and only if $GM\neq 0$.

Since there are no non-zero maps from modules of slope greater than zero to $M$, for all $g\in (P_2\oplus T_2\oplus Q_2 ,M)$, $g=g\mu_2\pi_2$. So the following diagram commutes.

\[\xymatrix@=20pt{
  (P_2\oplus T_2\oplus Q_2,M) \ar[d]_{(\mu_2,M)} \ar[r]^{(f,M)}
                & (P_1\oplus T_1\oplus Q_1,M) \ar[d]^{(\mu_1,M)}  \\
  (P_2\oplus T_2,M)  \ar[r]_{(\pi_2f\mu_1,M)}
                & (P_1\oplus T_1,M)             }\]

Moreover,  $(\mu_1,M)$ and $(\mu_2,M)$ are isomorphisms. Thus $g\in (P_1\oplus T_1\oplus Q_1,M)$ is in the image of $(f,M)$ if and only if $g\mu_1$ is in the image of $(\pi_2f\mu_2,M)$. So for $M$ a module of slope zero, $GM\neq 0$ if and only if $FM\neq 0$.

Now in order to show that there is some $L\in\ind\text{-}R$ of slope zero such that $FL\neq 0$, it is enough to show that $GL\neq 0$.

By the first paragraph, there exists $L\in\mod\text{-}R$ of slope greater than or equal to zero, such that $GL\neq 0$. Suppose that the slope of $L$ is greater than zero and let $h\in (P_1\oplus T_1,L)$ be such that it doesn't factor through $g$. Since the finite-dimensional modules of slope zero separate the preprojective modules from those of slope greater than zero, $h$ factors through some direct sum of finite-dimensional modules of slope zero. One of these modules $T$ is such that $GT\neq 0$.
\end{proof}

We gather together the facts we need in order to show that every finite-dimensional module of slope zero over a canonical algebra of tubular type is either in the image of $F_1:\Mod\text{-}A\rightarrow \Mod\text{-}A[X]$ or in the image of $F_2:\Mod\text{-}A\rightarrow \Mod\text{-}A[X]$.

From now on, let $m$ be the rank of the tube $\mcal{T}$ containing $X$. The quasi-simples at the mouth of $\mcal{T}$ are $X, \tau^{-1}X,\ldots, \tau^{-(m-1)}X$. Note that if $1\leq p<m$ and $i\in \N$ then $\Hom_A(X,\tau^{-p}X[i])=0$. So, in particular $F_1\tau^{-p}X[i]\cong F_0\tau^{-p}X[i]$.

A finite-dimensional module $M$ over a tame hereditary algebra is \textbf{regular} if $(M,P)=0$ for all preprojective $P$ and $(E,M)=0$ for all preinjective $E$.

\begin{lemma}\label{restrictiontoAregular}
If $M=(M_0,M_1,\Gamma_M)\in\Mod\text{-}A[X]$ is finite-dimensional, indecomposable and has slope zero then $M_1$ is regular. Moreover either $M=(0,M_1,0)$ with $M_1$ indecomposable (and regular) or $M_1$ is a sum of finite-dimensional modules of the form $X[i]$.
\end{lemma}
\begin{proof}
Suppose $P$ is a preprojective $A$-module, so $\Hom(X,P)=0$. Then $\Hom_A(rM,P)\cong\Hom_{A[X]}(M,F_1P)=0$ since $F_1P=(0,P,0)\in \mcal{P}_0$.

Note that if $Q$ is preinjective over $A$ then $(0,Q,0)$ has slope greater than zero. This is because $\langle\underline{\dim}X,\underline{\dim}Q\rangle =\dim\Hom_A(X,Q)-\dim\Ext_A(X,Q)=\dim\Hom_A(X,Q)\neq 0$, since $X$ is a regular module. Then $\Hom_A(Q,rM)\cong\Hom_{A[X]}(F_0Q,M)=0$ since $F_0Q\in \mcal{Q}_0$.

So for all preprojective $A$-modules $P$, $\Hom_A(rM,P)=0$ and for all preinjective $A$-modules $Q$, $\Hom_A(Q,rM)=0$. Thus $rM=M_1$ is regular.

For the final part, suppose that $M_1=L\oplus K$ where $L$ is a direct sum of $A$-modules of the form $X[i]$ and $K$ is a direct sum of regular modules not of the form $X[i]$. Then $\Hom_A(X,K)=0$. So $(M_0,M_1,\Gamma_M)$ decomposes as a direct sum of $(0,K,0)$ and $(M_0,L,\Gamma_M)$.
\end{proof}

\begin{lemma}\label{Easyass}

\begin{enumerate}
\item When $m\geq 2$, for each $i\in\N$ and $1\leq p<m$,
\[\xymatrix@C=0.38cm{
  0 \ar[r] & \scalebox{0.85}{$F_0\tau^{-p}X[i]$} \ar[rr] && \scalebox{0.85}{$F_0\tau^{-p}X[i+1]\oplus F_0\tau^{-(p+1)}X[i-1]$} \ar[rr] && \scalebox{0.85}{$F_0\tau^{-(p+1)}X[i] \ar[r]$} & 0 }\] is an almost split exact sequence, where $\tau^{-(p+1)}X[0]=0$.
\item For each $i\in \N$,
\[\xymatrix@C=0.4cm{
  0 \ar[r] & F_1X[i] \ar[rr] && F_1X[i+1]\oplus F_0\tau^{-1}X[i-1] \ar[rr] && F_0\tau^{-1}X[i] \ar[r] & 0 }\] is an almost split exact sequence, where $\tau^{-1}X[0]=0$.
\end{enumerate}
\end{lemma}
\begin{proof}
Apply \cite[XV 1.6]{ss3} to the almost split sequence $0\rightarrow \tau^{-p}X[i]\rightarrow \tau^{-p}X[i+1]\oplus \tau^{-(p+1)}X[i-1]\rightarrow \tau^{-(p+1)}X[i]\rightarrow 0$.
\end{proof}

%

The following lemma is most likely well known but since we couldn't find a reference, we include a proof.

\begin{lemma}\label{otherass}
For each $i\in\N$,
\[\xymatrix@C=0.5cm{
  0 \ar[r] & F_0X[i] \ar[rr] && F_1X[i]\oplus F_0X[i+1] \ar[rr] && F_1X[i+1] \ar[r] & 0 }\] is an almost split exact sequence.
\end{lemma}
\begin{proof}
We prove this by induction on $i$.

Suppose $i=1$. First note that the embedding of $X[1]$ into $X[2]$ remains irreducible in $A[X]$ after applying $F_0$ and the canonical embedding of $F_0X[1]$ into $F_1X[1]$ is irreducible since it is the embedding of the radical of an indecomposable projective.

Suppose that $N$ is an indecomposable non-projective module and that $f:F_0X[1]\rightarrow N$ is irreducible. By \cite[IV 3.8]{Ass1}, there exists an irreducible map from $\tau N$ to $F_0X[1]$. So by \ref{Easyass}, $\tau N\cong F_1X[2]$ if $m=1$ and $\tau N = F_0\tau^{-(m-1)}X[2]$ otherwise. In either case \ref{Easyass} implies $N\cong F_0X[2]$. It now remains to remark that $\Hom_{A[X]}(F_0X[1],F_1X[1])\cong\Hom_A(X[1],X[1])$ and $\Hom_{A[X]}(F_0X[1],F_0X[2])\cong\Hom_A(X[1],X[2])$ are both one dimensional and that the cokernel of the left minimal almost split map from $F_0X[1]$ to $F_1X[1]\oplus F_0X[2]$ is $F_1X[2]$.

Now suppose that we have proved the assertion of the lemma for all $i\leq n$. Suppose that $N$ is an indecomposable non-projective module and that $f:F_0X[n+1]\rightarrow N$ is irreducible. Then, as before, there is an irreducible map from $\tau N$ to $F_0X[n+1]$. So, by \ref{Easyass}, if $m=1$ then $\tau N\cong F_0X[n]$ or $\tau N\cong F_1X[n+2]$, and, if $m\neq 1$ then $\tau N\cong F_0 X[n]$ or $\tau N\cong F_0\tau^{m-1}X[n+2]$. If $\tau N\cong F_0 X[n]$ then, by the induction hypothesis, $N\cong F_1 X[n+1]$. If $m=1$ and $\tau N\cong F_1X[n+2]$, or, if $m\neq 1$ and $\tau N\cong F_0\tau^{m-1}X[n+2]$ then $N\cong F_0X[n+2]$.

It remains now to note that every map from $F_0X[n+1]$ to $F_1X[n+1]$ factors though the canonical embedding and that the spaces of irreducible morphisms $\text{Irr}_{A[X]}(F_0X[n+1],F_0X[n+2])$ and $\text{Irr}_{A}(X[n+1],X[n+2])$ are isomorphic.
\end{proof}

\begin{lemma}\label{GammaMembedding}
If $M\in\Mod\text{-}A[X]$ is indecomposable and not injective then $\Gamma_M$ is an embedding.
\end{lemma}
\begin{proof}
Note that $(M_0,M_1,\Gamma_M)\cong (\ker\Gamma_M,0,0)\oplus (M_0/\ker\Gamma_M,M_1,\Gamma_M)$ and the only indecomposable $A[X]$-module with $M_1=0$ is the simple injective module.
\end{proof}

\begin{proposition}\label{fdinimF0F1}
If $M$ is an indecomposable finite-dimensional module of slope zero then $M =F_0N$ or $M=F_1N$ for some indecomposable regular module $N$.
\end{proposition}
\begin{proof}
Since $M:=(M_0,M_1,\Gamma_M)$ is not injective either $M_0=0$ and $M$ is in the image of $F_0$, or, $M_0\neq 0$ and, using \ref{GammaMembedding} and the adjunction $\Hom_A(rM,M_1)\cong \Hom_{A[X]}(M,F_1M_1)$, there is an embedding of $M$ into $F_1M_1$. From \ref{restrictiontoAregular} we know that if $M_0\neq 0$ then $F_1M_1$ is a direct sum of modules of the form $F_1X[i]$. Thus there is a non-zero map $f=(f_0,f_1)$ from $M$ to some $F_1X[i]$. Take $i$ minimal such that $f_0\neq 0$. First suppose $i>1$. If $M$ is not isomorphic to $F_1X[i]$ then $f$ factors through the right minimal almost split map from $F_0X[i]\oplus F_1X[i-1]$ as in \ref{otherass}. Since $f_0\neq 0$ there is a non-zero map from $M$ to $F_1 X[i-1]$ contradicting the minimality of $i$. Thus $M\cong F_1X[i]$.

Now suppose $i=1$. Then $F_1X[1]$ is an indecomposable projective and $F_0X[1]$ is its radical. Thus either $f:M\rightarrow F_1X[1]$ is surjective or $f$ factors through $F_0X[1]$. If $f:M\rightarrow F_1X[1]$ is surjective then it is split since $F_1X[1]$ is projective. So, since $M$ is indecomposable, $M\cong F_1X[1]$. The second possibility can't occur since $f_0\neq 0$.
\end{proof}

\begin{proposition}
Every indecomposable pure-injective module of slope zero is either in the image of $F_0$ or in the image of $F_1$. Note that all preprojective modules are in the image of $F_0$.
\end{proposition}
\begin{proof}
The pure-injective modules of slope zero form a closed subset $\mcal{C}_0$ of the Ziegler spectrum. By \ref{fddenseinslopezero}, the finite-dimensional indecomposable modules of slope zero are dense in this set. We have shown, \ref{F0intfunctorwithdefimage}, \ref{F1intfunctorwithdefimage}, that the images of $F_0$ and $F_1$ are definable subcategories. Let $\mcal{A}_0$ and $\mcal{A}_1$ be their images in $\Zg_{A[X]}$ intersected with $\mcal{C}_0$, note that both $\mcal{A}_0$ and $\mcal{A}_1$ are closed. Since all finite-dimensional points of slope zero are contained in either $\mcal{A}_0$ or $\mcal{A}_1$ the closure of $\mcal{A}_0\cup\mcal{A}_1$ is $\mcal{C}_0$. Thus $\mcal{A}_0\cup\mcal{A}_1=\mcal{C}_0$ as required.
\end{proof}

We now use the above to provide an algorithm which given pp-pairs
\[
\phi/\psi,\phi_1/\psi_1,\ldots,\phi_n/\psi_n
\]
answers whether there is a preprojective or module of slope zero in $(\phi/\psi)$ but not in $\bigcup_{i=1}^n(\phi_i/\psi_i)$.

\begin{proposition}\label{algforslopezero}
Let $A$ be a tame hereditary algebra and $A[X]$ be a canonical algebra of tubular type both over a recursive algebraically closed field. There is an algorithm which given pp-pairs $\phi/\psi,\phi_1/\psi_1,\ldots,\phi_n/\psi_n$ answers whether there is an indecomposable pure-injective module $N$ such that $N\in\left(\phi/\psi\right)\cap(\mcal{E}_0\cup\mcal{E}_1)$ and $N\notin \bigcup_{i=1}^n \left(\phi_i/\psi_i\right)\cap(\mcal{E}_0\cup\mcal{E}_1)$.
\end{proposition}
\begin{proof}
If $N\in \mcal{P}_0\cup\mcal{C}_0$ then $N\in \mcal{E}_0$ or $N\in \mcal{E}_1$. We now describe how to effectively check whether there is an $N\in \mcal{E}_0$ such that $N\in\left(\phi/\psi\right)$ and $N\notin \bigcup_{i=1}^n\left(\phi_i/\psi_i\right)$. Given a pp-pair $\phi/\psi$ we can effectively translate this to a pp-pair $\sigma/\tau$ in the language of $A$-modules via $F_0$ such that $N\in \left(\sigma/\tau\right)$ if and only if $F_0N\in(\phi/\psi)$. Since the common theory of modules over a tame hereditary algebra is decidable (\cite{Geislerthesis}) we can answer whether one Ziegler basic open set over $A$ is contained in a finite union of other specified Ziegler open sets over $A$.

The argument is exactly the same for $\mcal{E}_1$.
\end{proof}

We now deal with the preinjective component and modules of slope $\infty$. If $R$ is a canonical algebra of tubular type (respectively a tubular algebra) then $R^{\text{op}}$ is also a canonical algebra of the same tubular type (respectively a tubular algebra) as $R$. For canonical algebras, this can be easily seen from the original definition of canonical algebra \cite[pg 161]{Ringeltub} and for tubular algebras is \cite[5.2.3]{Ringeltub}.

We have shown \ref{fddenseinslopezero} that, for a tubular algebra $R$, the finite-dimensional indecomposable $R$-modules of slope zero are dense in $\mcal{C}_0$. Using elementary duality, this implies the same result for $\mcal{C}_\infty$.

\begin{proposition}\label{fddenseinfinity}
Let $R$ be a tubular algebra. The finite-dimensional indecomposable $R$-modules of slope infinity are dense in the definable subcategory of modules of slope infinity.
\end{proposition}
\begin{proof}
Suppose that $N\in \mcal{C}_\infty$ and $N\in\left(\phi/\psi\right)$. By \ref{kdualandelemdual} and \ref{dualityC0Cinf}, $\Hom(N,k)\in \mcal{D}_0$ and $\Hom(N,k)\in \left(D\psi/D\phi\right)$. Thus, by \ref{fddenseinslopezero}, there is an indecomposable finite-dimensional module $M$ of slope zero such that $M\in\left(D\psi/D\phi\right)$. Now $\Hom(M,k)\in \left(\phi/\psi\right)$ and is an indecomposable finite-dimensional module of slope infinity.
\end{proof}

If $\mcal{E}$ is a definable subcategory of $\Mod\text{-}R$ such that $\mcal{E}:=\{N\in\Mod\text{-}R \st \phi_i(N)=\psi_i(N) \text{ for all } i\in I\}$ then let $D\mcal{E}$ be the definable subcategory of $\Mod\text{-}R^{\text{op}}$ such that $D\mcal{E}:=\{N\in\Mod\text{-}R^{\text{op}} \st D\phi_i(N)=D\psi_i(N) \text{ for all } i\in I\}$.

\begin{lemma}
Let $A$ be a tame hereditary algebra and $R:=A[X]$ be a canonical algebra of tubular type. Every indecomposable pure-injective module of slope infinity and every indecomposable preinjective module over $R^{\text{op}}$ is in $D\mcal{E}_0\cup D\mcal{E}_1$.
\end{lemma}
\begin{proof}
This is true for all finite-dimensional modules by \ref{kdualslope}.
So by \ref{fddenseinfinity}, this is also true for all indecomposable pure-injectives of slope infinity.
\end{proof}

If $R$ is a canonical algebra of tubular type then let $\mcal{E}'_0$ (respectively $\mcal{E}'_1$) be $D\mcal{E}_1$ (respectively $D\mcal{E}_0$) where $\mcal{E}_0$ and $\mcal{E}_1$ are the images of $F_0$ (respectively $F_1$) as functors to $\Mod\text{-}R^{\text{op}}$.

\begin{lemma}\label{Herdual0toinf}
Let $A$ be a tame hereditary algebra and $R:=A[X]$ be a canonical algebra of tubular type. Let $\phi/\psi,\phi_1/\psi_1,\ldots ,\phi_n/\psi_n$ be pp-pairs over $R$.
There exists an indecomposable pure-injective $R$-module $N\in\mcal{E}_0\cup\mcal{E}_1$ such that
$N\in \left(\phi/\psi\right)$ and $N\notin \bigcup\left(\phi_i/\psi_i\right)$
if and only if there exists a indecomposable pure-injective  $R^{\text{op}}$-module $L\in \mcal{E}'_0\cup \mcal{E}'_1$ such that
$L\in \left(D\psi/D\phi\right)$ and $L\notin \bigcup\left(D\psi_i/D\phi_i\right)$.
\end{lemma}
\begin{proof}
Suppose $L\in\mcal{E}'_0\cup \mcal{E}'_1$ is such that $L\in\left(\phi/\psi\right)$ and $L\notin\bigcup_{i=1}^n\left(\phi_i/\psi_i\right)$.

Then $\Hom(L,k)$ is in the definable category $\mcal{E}_0\cup\mcal{E}_1$ and $\Hom(L,k)$ opens $D\psi/D\phi$ but not $D\psi_i/D\phi_i$ for any $1\leq i\leq n$. Therefore there is an indecomposable pure-injective module $M$ over $R^{\text{op}}$ in $\mcal{E}'_0\cup \mcal{E}'_1$ which is in $\left(D\psi/D\phi\right)$ but not in $\bigcup_{i=1}^n\left(D\psi_i/D\phi_i\right)$.

The reverse direction is proved symmetrically.
\end{proof}

\begin{cor}\label{algforslopeinfinity}
Let $A$ be a tame hereditary algebra and $A[X]$ be a canonical algebra of tubular type over a recursive algebraically closed field. There is an algorithm which given pp-pairs $\phi/\psi,\phi_1/\psi_1,\ldots,\phi_n/\psi_n$ answers whether there is an indecomposable pure-injective module $N$ such that $N\in\left(\phi/\psi\right)\cap(\mcal{E}'_0\cup \mcal{E}'_1)$ and $N\notin \bigcup_{i=1}^n \left(\phi_i/\psi_i\right)\cap(\mcal{E}'_0\cup \mcal{E}'_1)$.
\end{cor}


\section{Corrections to a paper of Harland and Prest}\label{corharlandprest}
Throughout this section, unless explicitly indicated, $R$ will be a tubular algebra.

The main work of this section is to show, \ref{accallslopes}, that Corollary 8.8 of \cite{modirrslope} is false for all tubular algebras and to provide, \ref{replcor}, a best possible replacement. Although the replacement of corollary 8.8 will not be used in later sections, many statements in this section will be needed.

We start by correcting some statements in section 3 of \cite{modirrslope}.



In \cite{modirrslope}, it is claimed that for $a,b\in\R^+$ the set of modules which are direct limits of finite-dimensional modules with slope in the interval $(a,b)$ is a definable subcategory of $\Mod\text{-}R$, in \cite{modirrslope} this is called the set of modules supported on $(a,b)$. This is false for $a,b\in\Q^+$. The problem is, that although the set of modules supported on $(a,b)$ is a definable category by \cite[2.1]{LenzingHomtransfer}, it is not a definable subcategory of $\Mod\text{-}R$.

Using the terminology of \cite{modirrslope}, the set of modules lying over $(a,b)$, i.e. those modules $M$ such that $M\otimes P^*=0$ for all finite-dimensional $P$ of slope less than or equal to $a$ and $(P,M)=0$ for all finite-dimensional $P$ of slope greater than or equal to $b$, is a definable subcategory by definition. However, the description of the indecomposable pure-injectives lying over $(a,b)$ given in \cite{modirrslope} is not correct.  The following proposition corrects this.

\begin{proposition}\label{descriptionofDab} Let $R$ be a tubular algebra and let $a,b\in \R^+$. The smallest definable subcategory, $\mcal{D}^+_{(a,b)}$, containing all finite-dimensional indecomposable modules with slope in $(a,b)$ contains exactly all indecomposable pure-injectives with slope in $(a,b)$ plus,
 \begin{enumerate}
 \item the Pr\"ufer and generic modules of slope $a$ if $a\in\Q^+$
 \item the adic and generic modules of slope $b$ if $b\in \Q^+$
 \item all indecomposable pure-injective modules of slope $a$ if $a\notin \Q^+$
 \item all indecomposable pure-injective modules of slope $b$ if $b\notin \Q^+$
 \end{enumerate}
\end{proposition}

Before we prove the proposition, we need a few lemmas and to recall a few facts. The following remark will hold for general rings if $\Hom(M,k)$ is replaced an appropriate notion of dual module (see \cite{PSL} for notions of dual modules in the general context). We will only need it for finite-dimensional $k$-algebras.

\begin{remark}
Let $R$ be a finite-dimensional $k$-algebra. If $\{M_i \st i\in I\}$ is a set of finite-dimensional left modules and $L\in\langle M_i \st i\in I \rangle$ then $\Hom(L,k)\in\langle \Hom(M_i,k) \st i\in I\rangle$.
\end{remark}

\begin{lemma}\label{directunions}\cite[3.4]{modirrslope}
Suppose that $M$ is an indecomposable module of positive slope $r > 0$. Then for every $\epsilon > 0$, $M$
is the directed union of its finite-dimensional submodules in $(r - \epsilon,r]$, indeed in $(r - \epsilon,r)$ in
the case that r is irrational.
\end{lemma}

\begin{lemma}\label{homsbetweeninfandfdatq}
Let $M$ be a finite-dimensional indecomposable module of slope $q\in \Q^+$ and let $E$ be a quasi-simple of slope $q$. Then
\begin{enumerate}[(i)]
\item $\Hom(\mcal{G}_q,M)=0$
\item $\Hom(E[\infty],M)=0$
\item $\Hom(M, \mcal{G}_q)=0$
\item $\Hom(M,\widehat{E})=0$
\end{enumerate}
\end{lemma}
\begin{proof}
The first two statements are a consequence of \cite[Theorem 6.4]{InfdimcanalgReitenRingel}. The right adic modules of slope $q$ are $k$-duals of the left Pr\"ufer modules at slope $1/q$, the generic module of slope $q$ is the $k$-dual of the generic module of slope $1/q$ and $k$-duality sends finite-dimensional indecomposable modules of slope $q$ to finite-dimensional modules of slope $1/q$. Thus, if $\Hom(M,\widehat{E})\neq 0$ (respectively $\Hom(M,\mcal{G}_q)\neq 0$) then $\Hom(E^*[\infty],M^*)\neq 0$ (respectively $\Hom(\mcal{G}_{1/q},M^*)\neq 0$). But this can't happen by parts (i) and (ii).
\end{proof}

\begin{proposition}\label{nearirrational}\cite[3.11]{modirrslope}
Let $\phi/\psi$ be a pp-pair and $r\in \R^+$ irrational. Then the following are equivalent:

\begin{enumerate}
\item there is an $\epsilon>0$ such for all finite-dimensional modules $M$ lying in homogeneous tubes with slope in $(r,r+\epsilon)$, $\phi(M)>\psi(M)$.
\item there is an $\epsilon>0$ such for all finite-dimensional modules $M$ lying in homogeneous tubes with slope in $(r-\epsilon,r)$, $\phi(M)>\psi(M)$.
\item there is an indecomposable pure-injective module $N$ with slope $r$ such that $\phi(N)>\psi(N)$.
\end{enumerate}
\end{proposition}

\begin{proof}[Proof of \ref{descriptionofDab}]
We give proofs for the case $a,b\in\Q^+$ and the case $a,b\notin\Q^+$.

Note that the definable subcategory $\mcal{D}^+_{(a,b)}$ is contained in the definable subcategory of modules $M$ such that $(P,M)=0$ for all finite-dimensional $P$ of slope greater than or equal to $b$ and $M\otimes P^*=0$ for all finite-dimensional $P$ of slope less than or equal to $a$. We will now show that this definable subcategory is in fact $\mcal{D}^+_{(a,b)}$.

By \ref{directunions}, every indecomposable pure-injective module with slope in $(a,b)$ is in $\mcal{D}_{(a,b)}^+$.

If $b\in\Q^+$, $\epsilon\in \R^+$ and a definable subcategory $\mcal{D}$ of $\Mod\text{-}R$ contains all finite-dimensional indecomposable modules with slope in $(b-\epsilon,b)$ then $\mcal{D}$ contains the generic at $b$ and all adic modules at $b$. This is because any module of slope $b$ is a direct union of direct sums of indecomposable finite-dimensional modules with slope in $(b-\epsilon,b]$ and no finite-dimensional indecomposable module of slope $b$ is a submodule of the generic at $b$ or any adic module at $b$ by \ref{homsbetweeninfandfdatq}.

If $a\in\Q^+$, $\epsilon\in \R^+$ and a definable subcategory $\mcal{D}$ of $\Mod\text{-}R$ contains all finite-dimensional indecomposable modules with slope in $(a,a+\epsilon)$ then $\mcal{D}$ contains the generic at $a$ and all Pr\"{u}fer modules at $a$. This follows from the above paragraph, since each left adic module at $1/a$ is equal to the $k$-dual of some right Pr\"ufer module at $a$, so, by \cite[1.3.16]{PSL}, every right Pr\"ufer module at $a$ is a pure-submodule of the $k$-dual of a left adic module at $1/a$.

Note now that, by \ref{closedsubsets}, if any definable subcategory contains either a Pr\"ufer module at slope $c$ or an adic module at slope $c$ then it also contains the generic module at slope $c$.

Thus if $a, b$ are both rational then $\mcal{D}_{(a,b)}^+$ contains all pure-injective indecomposables with slope in $(a,b)$, the adic and generic modules at $b$ and the Pr\"ufer and generic modules at $a$. These are exactly the pure-injective indecomposable modules $M$ such that $(P,M)=0$ for all finite-dimensional $P$ of slope greater than or equal to $b$ and $M\otimes P^*=0$ for all finite-dimensional $P$ of slope less than or equal to $a$.

If both $a,b$ are irrational then the situation is much simpler. All indecomposable modules of slope $b$ are direct unions of direct sums of finite-dimensional modules with slope in $(a,b)$. In order to deal with the $a$ irrational case we use \ref{nearirrational}, which says that if $\phi/\psi$ is open on some indecomposable pure-injective module of slope $a$ then it is open on all homogeneous tubes with slope in $(a,a+\epsilon)$ for some $\epsilon>0$.  Thus if $\phi/\psi$ is open on some indecomposable pure-injective module of slope $a$ then it is open on some finite-dimensional indecomposable module with slope in $(a,b)$. So the indecomposables pure-injectives of slope $a$ are in the definable subcategory generated by the finite-dimensional indecomposable modules with slope in $(a,b)$.
\end{proof}

\begin{definition}
For $a,b\in \Q^+$, let $\mcal{C}_{(a,b)}:=\mcal{D}^+_{(a,b)}\cap\pinj_R$.
\end{definition}

\subsection{A replacement for corollary 8.8 of \cite{modirrslope}}

We now consider corollary 8.8 of \cite{modirrslope} and prove a replacement.

We say a pp-pair $\phi/\psi$ is \textbf{uniformly open} at $q\in\Q^+$ if $\phi/\psi$ is open on all finite-dimensional indecomposable modules of slope $q$. We say that $\phi/\psi$ is \textbf{uniformly closed} at $q\in\Q^+$ if $\phi/\psi$ is closed on all finite-dimensional indecomposable modules of slope $q$. We say that $q\in\Q^+$ is a \textbf{non-uniform slope} for $\phi/\psi$ if $\phi/\psi$ is neither uniformly open or closed at $q$.

Corollary 8.8 of \cite{modirrslope} states that if $\phi/\psi$ is a pp-pair over a tubular algebra then for all but finitely many $r\in \R^+$, $\phi/\psi$ is either $\phi/\psi$ is open on all indecomposable pure-injective modules of slope $r$ or closed on all indecomposable pure-injective modules of slope $r$. It further states that the set of $r\in\R^+$ for which $\phi/\psi$ is open on all indecomposable pure-injective modules of slope $r$ is the union of finitely many rational points and intervals with rational endpoints. We will show in \ref{accallslopes} that for all tubular algebras this is not the case and in fact, that for all $p\in \Q_0^\infty$ there exists a pp-pair $\phi/\psi$ such that $p$ is an accumulation point of the set of slopes $q\in\Q^+$ where $\phi/\psi$ is neither uniformly open nor uniformly closed at $q$.

We first prove the following which is a best possible replacement for corollary 8.8 of \cite{modirrslope}.

\begin{theorem}\label{replcor}
Let $\phi/\psi$ be a pp-pair and $S$ be the set of slopes $q\in\Q^+$ where $\phi/\psi$ is neither uniformly open nor uniformly closed at $q$. The set $S$ has finitely many accumulation points in $\R$, and all these accumulation points are in $\Q$.
\end{theorem}

The following series of lemmas will be used in the proof of \ref{replcor}.

\begin{lemma}\label{openonhomoimpliesoenonallbutfinitelymany}
If $q\in \Q^+$ and $\phi/\psi$ is open on all finite-dimensional modules of slope $q$ in homogeneous tubes then $\phi/\psi$ is closed on at most finitely many $X\in \ind\text{-}R$ of slope $q$.
\end{lemma}
\begin{proof}
This follows directly from \ref{compactopenatq}.
%
%
%
\end{proof}

\begin{lemma}\label{closedonalloropenonallhomgeneous}
Suppose that $q\in\Q^+$, $\phi/\psi$ is a pp-pair and $v\in K_0(R)$ is such that $\dim \phi/\psi(X)=v\cdot\udim X$ for all $X\in\ind\text{-}R$ of slope $q$. Then $\phi/\psi$ is either open on all modules in homogeneous tubes of slope $q$ or closed on all modules in homogeneous tubes of slope $q$.
\end{lemma}
\begin{proof}
Let $w$ be the dimension vector of a finite-dimensional quasi-simple in a homogeneous tube of slope $q$. Then for all finite-dimensional indecomposable modules $X$ of slope $q$ lying in homogeneous tubes, $\udim X=n\cdot w$ for some $n\in\N$. Since $\dim \phi/\psi(X)=v\cdot\udim X$ for all $X\in\ind\text{-}R$ of slope $q$, $\phi/\psi$ is open on all modules in homogeneous tubes of slope $q$ if $v\cdot w>0$ and $\phi/\psi$ is closed on all modules in homogeneous tubes of slope $q$ if $v\cdot w=0$.
\end{proof}

\begin{proposition}\label{closedhomqimpliesclosedq}
 Suppose that $q\in\Q^+$, $\phi/\psi$ is a pp-pair and $v\in K_0(R)$ is such that $\dim \phi/\psi(X)=v\cdot\udim X$ for all $X\in\ind\text{-}R$ of slope $q$. If $\phi/\psi$ is closed on all modules of slope $q$ in homogeneous tubes then $\phi/\psi$ is closed on all modules of slope $q$.
\end{proposition}
\begin{proof}
Let $E_1,\ldots,E_n$ be the quasi-simples at the mouth of an inhomogeneous tube $T(\rho)$ of slope $q$. Then $\udim E_1+\ldots+\udim E_n$ is the dimension vector of an indecomposable module in a homogeneous tube with slope $q$. Thus $\phi/\psi$ is closed on all modules with dimension vector $\udim E_1+\ldots+\udim E_n$, so $v\cdot(\udim E_1+\ldots+\udim E_n)=0$. Since $v\cdot\udim E_i\geq 0$ for $1\leq i\leq n$, it follows that $v\cdot\udim E_i=0$ for $1\leq i\leq n$. Thus $v\cdot \udim X=0$ for every finite-dimensional module in $T(\rho)$. Thus $\phi/\psi$ is closed on all finite-dimensional indecomposable modules of slope $q$. So, by \ref{closedsubsets}, $\phi/\psi$ is closed on all modules of slope $q$.
\end{proof}

\begin{lemma}\label{uniformclosedrationininterval}
 Let $a<b\in\Q_0^\infty$, $\phi/\psi$ be a pp-pair and suppose there is a $v\in K_0(R)$ such that $\dim \phi/\psi(X)=v\cdot\udim X$ for all $X\in\ind\text{-}R$ of slope $q\in (a,b)$. If $\phi/\psi$ is closed on all homogeneous tubes of slope $q$ for some rational $q\in (a,b)$ then $\phi/\psi$ is uniformly closed on all rational slopes in $(a,b)$.
\end{lemma}
\begin{proof}
Let $q\in (a,b)$ be such that $\phi/\psi$ is closed on all homogeneous tubes of slope $q$. By \ref{closedhomqimpliesclosedq}, we may assume that $\phi/\psi$ is uniformly closed at $q$.

Thus
\[\mcal{C}_q=\bigcap_{\epsilon>0}\mcal{C}_{(q-\epsilon,q+\epsilon)}\subseteq \Zg_R\backslash (\phi/\psi).\]

So, since the closed sets $\mcal{C}_{(q-\epsilon,q+\epsilon)}$ form a chain and $(\phi/\psi)$ is compact, there exists a $\delta>0$ such that
\[(\phi/\psi)\subseteq \Zg_R\backslash\mcal{C}_{(q-\delta,q+\delta)}.\] Thus $\phi/\psi$ is closed on $\mcal{C}_{(q-\delta,q+\delta)}$. This implies that $v\cdot(ch_0+dh_{\infty})=0$ for all $d/c\in(q-\delta,q+\delta)$. Thus $v\cdot h_0=v\cdot h_\infty=0$. So $v\cdot \udim X=0$ for all $X$ in homogeneous tubes with slope in $(a,b)$ and thus $\phi/\psi$ is uniformly closed on all rational slopes in $(a,b)$.
\end{proof}

\begin{lemma}\label{opensetaround}
 If $p\in \Q^+$ is such that $\phi/\psi$ is closed on just finitely many indecomposable modules of slope $p$ and $p$ is rational then there is an $\epsilon>0$ such that for all $q\in (p-\epsilon,p+\epsilon)\backslash\{p\}$, $\phi/\psi$ is uniformly open at $q$.
\end{lemma}
\begin{proof}
Let $\{Z_1,\ldots,Z_n\}$ be the indecomposable modules of slope $p$ which do not open $\phi/\psi$. Let $\sigma/\tau$ be a pp-pair such that $(\sigma/\tau)=\{Z_1,\ldots,Z_n\}$. Then $\mcal{U}:=(\phi/\psi)\cup (\sigma/\tau)$ is open and \[\mcal{U}\cap\bigcap_{\epsilon>0}\mcal{C}_{(p-\epsilon,p+\epsilon)}=\mcal{U}\cap\mcal{C}_p=\mcal{C}_p=\bigcap_{\epsilon>0}\mcal{C}_{(p-\epsilon,p+\epsilon)}.\]

Thus $\Zg_R\backslash\mcal{U}\subseteq\cup_{\epsilon>0}\Zg_R\backslash \mcal{C}_{(p-\epsilon,p+\epsilon)}$. Since $\Zg_R\backslash \mcal{U}$ is closed and hence compact, there is an $\epsilon>0$ such that $\Zg_R\backslash \mcal{U}\subseteq \Zg_R\backslash\mcal{C}_{(p-\epsilon,p+\epsilon)}$. So $\mcal{C}_{(p-\epsilon,p+\epsilon)}\subseteq \mcal{U}$. Thus for all $q\in (p-\epsilon,p+\epsilon)\backslash\{p\}$, $\phi/\psi$ is uniformly open at $q$.
\end{proof}

The following is inspired by \cite[Theorem 3.2]{modirrslope}.

\begin{proposition}
Let $\phi$ be a pp-formula. There is an algorithm which outputs $n\in\N$, $q_0=0<q_1<q_2<\ldots<q_n<q_{n+1}=\infty\in\Q_0^\infty$ and $v_1,\ldots,v_{n+1}\in K_0(R)$ such that for $0\leq i\leq n$, for all finite-dimensional indecomposable modules $N$ with slope in $(q_i,q_{i+1})$, \[\dim\phi(N)=v_i\cdot \udim N.\]
\end{proposition}
\begin{proof}
There is an algorithm, \ref{recfreerealise}, which given $\phi$ outputs a presentation $(k^n,A_1,\ldots,A_s)$ of a finite-dimensional module $M$ and an element $m\in k^n$ such that $(M,m)$ is a free-realisation of $\phi$. Note that, \cite[Corollary 1.2.19]{PSL} for any finite-dimensional module $L$, $\dim\phi(L)=\dim \Hom(M,L)-\dim\Hom(\coker(m),L)$.

There is an algorithm, \ref{effdecomp}, which given a presentation of $M$ outputs presentations of its indecomposable factors with multiplicity. From this we can compute the dimension vectors of the indecomposable factors of $M$ and $\coker(m)$ with multiplicity. We may now, by \ref{effdimvector}, compute the slope of each of the indecomposable factors of $M$ and $\coker(m)$. Let $q_1<\ldots<q_n$ be the slopes of those indecomposable factors which have slope greater than zero and less than infinity.

For $0\leq i\leq n$, let $w_i$ (respectively $u_i$) be the sum of the dimension vectors of all indecomposable factors of $M$ (respectively of $\coker(m)$) with slope strictly smaller than $q_{i+1}$ (equivalently have slope less than or equal to $q_i$).

Since all indecomposable factors of $M$ are either preprojective or have slope less than or equal to $q_i$,
\[\dim\Hom(M,N)=\langle w_i, \udim N\rangle+\dim \Ext(M,N)=\langle w_i, \udim N\rangle\] and
\[\dim\Hom(\coker(m),N)=\langle u_i, \udim N\rangle+\dim \Ext(M,N)=\langle u_i, \udim N\rangle\] for all finite-dimensional indecomposable $N$ with slope in $(q_{i},q_{i+1})$.

Thus $\dim\phi(N)/\psi(N)=\langle w_i-u_i, \udim N\rangle$ for $N$ with slope in $(q_{i},q_{i+1})$. For $0\leq i\leq n$, let $v_i=(\langle w_i-u_i, \udim S_1\rangle,\ldots,\langle w_i-u_i, \udim S_m\rangle)$ where $S_1,\ldots,S_m$ are the simple modules over $R$.
\end{proof}

\begin{cor}\label{dimvectorsforpppairs}
Let $\phi/\psi$ be a pp-pair. There is an algorithm which outputs $n\in\N$, $q_0=0<q_1<q_2<\ldots<q_n<q_{n+1}=\infty\in\Q_0^\infty$ and $v_0,\ldots,v_{n}\in K_0(R)$ such that for all for all finite-dimensional indecomposable modules $N$ with slope in $(q_i,q_{i+1})$, \[\dim\phi/\psi(N)=v_i\cdot \udim N.\]
\end{cor}

\begin{proof}[Proof of theorem \ref{replcor}]
By \ref{dimvectorsforpppairs}, it is enough to show that if $a<b\in\Q^\infty_0$ and there exists $v\in K_0(R)$ such that for all $M\in\ind\text{-}R$,
\[\dim\phi(M)/\psi(M)=v\cdot\udim M\] then there are only finitely many accumulation points of non-uniform slopes for $\phi/\psi$ in $(a,b)$.

Lemmas \ref{openonhomoimpliesoenonallbutfinitelymany}, \ref{closedonalloropenonallhomgeneous}, \ref{closedhomqimpliesclosedq} and \ref{uniformclosedrationininterval} show that either $\phi/\psi$ is uniformly closed on all rational $q\in (a,b)$ or for each rational $q\in (a,b)$, $\phi/\psi$ is open on all but finitely many points of slope $q$, all of which are finite-dimensional.

If $\phi/\psi$ is uniformly closed on all rational $q\in (a,b)$ then $\phi/\psi$ is closed on all indecomposable pure-injectives with slope in $(a,b)$. This is because the finite-dimensional indecomposable modules are dense in $\mcal{C}_{(a,b)}$.

If for each rational $q\in (a,b)$, $\phi/\psi$ is open on all but finitely many points of slope $q$, all of which are finite-dimensional then there are no rational accumulation points in the set of non-uniform slopes for $\phi/\psi$ between $a$ and $b$ by \ref{opensetaround}.

It remains to show that if for each rational $q\in (a,b)$, $\phi/\psi$ is open on all but finitely many points of slope $q$, all of which are finite-dimensional then there are no irrational accumulation points of non-uniform slopes for $\phi/\psi$ inside $(a,b)$.  Note that if $\phi/\psi$ is closed on $X\in\ind\text{-}R$ then $X$ is in an inhomogeneous tube. We refer forward to \ref{Omegaexists}, which states that there is a finite set $\Omega$ of roots of $\chi_R$ such that for all $X\in\ind\text{-}R$ lying in inhomogeneous tubes  $\udim X = y+ w$ where $y\in\Omega$ and $w\in\rad \chi_R$. Let $g_1,g_2$ generate $\rad \chi_R$. Note that, since $\phi/\psi$ is open on all homogeneous tubes with slope in $(a,b)$ either $vg_1\neq 0$ or $vg_2\neq 0$. So if $\udim X= y+ \alpha g_1 +\beta g_2$ for $y\in \Omega$ and $\alpha,\beta\in\Z$, $X$ has slope in $(a,b)$ and $\phi/\psi$ is closed on $X$ then $0=v\cdot\udim X=\alpha v\cdot g_1+\beta v\cdot g_2+v\cdot y$.

For fixed $y\in \Omega$, we now consider the set of $\alpha,\beta\in\Z$ such that $\alpha v\cdot g_1+\beta v\cdot g_2+v\cdot y=0$. If $v\cdot g_1=0$ then $\beta$ is a fixed integer and if $v\cdot g_1\neq 0$ then $\alpha=\sigma \beta+\mu$ for some fixed $\sigma,\mu\in\Q$. In the first case there are fixed rationals $c,d,e,f$ such that the slope of $y+\alpha g_1+\beta g_2$ is of the form $(c\alpha+d)/(e\alpha +f)$. So as $\alpha$ tends to $\pm\infty$, the slope of $y+\alpha g_1+\beta g_2$ tends to a rational or $\pm\infty$. In the second case there are fixed rationals $c,d,e,f$ such that the slope of $y+\alpha g_1+\beta g_2$ is of the form $(c\beta +d)/(e\beta+f)$. So as $\beta$ tends to infinity, the slope of $y+\alpha g_1+\beta g_2$ tends to a rational or $\pm\infty$. Therefore, since $\Omega$ is finite, there are no irrational accumulation points of non-uniform slope for $\phi/\psi$ in $(a,b)$.
\end{proof}

In the above proof we could have replaced the final argument with the following argument using \ref{onepointuptoelemeqatirrational}.
We know that if $r\in(a,b)$ is irrational and an accumulation point of non-uniform slopes for $\phi/\psi$ then $\phi/\psi$ is open on all points in $\mcal{C}_r$. Thus $\bigcap_{\epsilon>0}\mcal{C}_{(r-\epsilon,r+\epsilon)}=\mcal{C}_r\subseteq (\phi/\psi)$. So $\Zg_R\backslash (\phi/\psi)\subseteq \bigcup_{\epsilon>0} \Zg_R\backslash\mcal{C}_{(r-\epsilon,r+\epsilon)}$. Since $\Zg_R\backslash (\phi/\psi)$ is closed, it is compact. Thus there exists an $\epsilon_0>0$ such that $\Zg_R\backslash (\phi/\psi)\subseteq \Zg_R\backslash\mcal{C}_{(r-\epsilon_0,r+\epsilon_0)}$. So $\mcal{C}_{(r-\epsilon_0,r+\epsilon_0)}\subseteq(\phi/\psi)$. So $r$ is not an accumulation point of non-uniform slopes for $\phi/\psi$.

\subsection{Accumulation points of non-uniform slopes}
The rest of this section will be spent proving the following result.

\begin{proposition}\label{accallslopes}
Let $R$ be a tubular algebra. For any $q'\in \Q^+\cup\{0\}$ there exists $L\in\ind\text{-}R$ of slope $q'$ such that $q'$ is an accumulation point of the set of non-uniform slopes for $\Hom(L,-)$.
\end{proposition}

If $X$ is finite-dimensional then the functor $\Hom(X,-)$ is equivalent to a functor given by a pp-pair $\phi/\psi$ (see \ref{effrepfun}). So \ref{accallslopes} implies that for all $q'\in\Q^+\cup\{0\}$ there is a pp-pair $\phi/\psi$ such that $q'$ is an accumulation point of the set of non-uniform slopes of $\phi/\psi$ thus contradicting corollary 8.8 of \cite{modirrslope}.

\begin{cor}
Let $R$ be a tubular algebra. For all $q\in \Q_0^\infty$ there exists a pp-pair $\phi/\psi$ such that $q$ is an accumulation point of the set of non-uniform slopes for $\phi/\psi$.
\end{cor}
\begin{proof}
For all $q\in \Q^+\cup\{0\}$, this follows directly from \ref{accallslopes}. The result for $q=\infty$ follows by combining duality for pp-formulas with \ref{kdualandelemdual} and \ref{kdualslope}.
\end{proof}

\subsubsection{Coherent sheaves on weighted projective lines}\label{cohsheaves}

We will use categories of coherent sheaves on a weighted projective lines of tubular type to prove \ref{accallslopes}. In order to introduce notation and for the readers convenience, we briefly review various features of categories of coherent sheaves on weighted projective lines. Our main references are \cite{GeiLen} and \cite{LenMeltub}.

Let $t\in \N$ be greater than $2$, $\mathbf{p}=(p_1,\ldots,p_t)$ be a $t$-tuple of strictly positive integers $p_i$, called a \textbf{weight sequence}, and $\mathbf{\lambda}=(\lambda_1,\ldots,\lambda_t)$ a $t$-tuple of pairwise distinct elements of $\mathbb{P}_1(k)$, called a \textbf{parameter sequence}, normalised so that $\lambda_1=\infty$, $\lambda_2=0$ and, if it exists $\lambda_3=1$. For $1\leq i\leq t$, we will refer to the point $\lambda_i\in \mathbb{P}_1(k)$ as an \textbf{exceptional} point (of \textbf{weight} $p_i$) and all other points in $\mathbb{P}_1(k)$ as \textbf{ordinary}.

For every pair $(\mathbf{p},\mathbf{\lambda})$, Geigle and Lenzing define, \cite[1.1,1.5]{GeiLen}, a weighted projective line $\mathbb{X}:=\mathbb{X}(\mathbf{p},\mathbf{\lambda})$. We will not give this definition but instead define the category of coherent sheaves on $\mathbb{X}(\mathbf{p},\mathbf{\lambda})$ purely in terms of $(\mathbf{p},\mathbf{\lambda})$.

Given a weight sequence $\mathbf{p}$, let $\mathbb{L}:=\mathbb{L}(\mathbf{p})$ be the abelian group on generators $x_1,\ldots, x_t$ with the relations
\[p_1x_1=p_2x_2=\cdots p_tx_t=:c.\] The \textbf{degree homomorphism} $\delta:\mathbb{L}\rightarrow \Z$ is defined on generators by $\delta(x_i):=p_i/p$ where $p$ is the lowest common multiple of $p_1,\ldots,p_t$.


The $\mathbb{L}(\mathbf{p})$-graded $k$-algebra $S:=S(\mathbf{p},\mathbf{\lambda})$ is the quotient of $k[X_1,\ldots,X_t]$ by the ideal generated by \[f_i:=X_i^{p_i}-X_2^{p_2}-\lambda_iX_1^{p_1},\] for $3\leq i\leq t$, with the $\mathbb{L}(\mathbf{p})$-grading given by assigning $X_i$, for $1\leq i\leq t$, degree $x_i$.

Let $\mod^{\mathbb{L}}\text{-}S$ denote the category of finitely generated $\mathbb{L}$-graded $S$-modules with morphisms given by $S$-linear maps of degree zero. Let $\mod_0^{\mathbb{L}}\text{-}S$ denote the full subcategory of all graded modules of finite length. The category of coherent sheaves, $\coh(\mathbb{X})$, on the weighted projective line $\mathbb{X}$, is equivalent to the category $\mod^{\mathbb{L}}\text{-}S/\mod_0^{\mathbb{L}}\text{-}S$ (see \cite[Serre's theorem]{GeiLen} and \cite[7.4]{GeiLenperp}). The \textbf{structure sheaf} $\mcal{O}$ is the image of $S$ in $\mod^{\mathbb{L}}\text{-}S/\mod_0^{\mathbb{L}}\text{-}S$.

The group $\mathbb{L}$ acts on $\mod^{\mathbb{L}}\text{-}S$ by grading shift i.e. for $x\in \mathbb{L}$ and $M\in \mod^{\mathbb{L}}\text{-}S$, $M(x)$ is defined to be the $\mathbb{L}$-graded $S$-module such that for all $y\in\mathbb{L}$, the homogeneous component of degree $y$,  $M(x)_y$, is equal to $M_{x+y}$. Since the $\mathbb{L}$-action on $\mod^{\mathbb{L}}\text{-}S$ fixes $\mod_0^{\mathbb{L}}\text{-}S$ as a subcategory, $\mathbb{L}$ acts on $\coh(\mathbb{X})$.

The category $\text{coh}(\mathbb{X})$ is a hereditary hom-finite $k$-category with Serre duality. In particular, \cite[2.2]{GeiLen}, for all $X,Y\in\coh(\mathbb{X})$,
\[\D \Ext(X,Y)\cong \Hom(Y, X(\omega))\] where $\omega:=(t-2)c-\sum_{i=1}^tx_i$ is called the \textbf{dualising element}. Moreover, $\coh(\X)$ has almost-split sequences and the Auslander-Reiten translate $\tau X$ of $X\in\coh(\X)$ is $X(\omega)$.

The Grothendieck group of $\coh(\mathbb{X})$, denoted $K_0(\mathbb{X}):=K_0(\coh(\mathbb{X}))$, is equipped with the Euler form $\langle -,- \rangle: K_0(\mathbb{X})\times K_0(\mathbb{X})\rightarrow \Z$ which is given on sheaves $X,Y\in\coh(\mathbb{X})$ by
\[\langle [X],[Y] \rangle=\dim \Hom(X,Y)-\dim\Ext(X,Y) .\]

The torsion-free objects in $\text{coh}(\mathbb{X})$, i.e. those without non-zero subobjects of finite length, are called \textbf{vector bundles}. Every object in $\text{coh}(\mathbb{X})$ decomposes as $V\oplus F$ where $V$ is a vector bundle and $F$ is finite length.

The subcategory, $\text{coh}_0(\mathbb{X})$, of finite length objects is uniserial and decomposes into connected components as $\coprod_{\lambda\in\mathbb{P}^1(k)}\mcal{U}_\lambda$ where for each $\lambda\in \mathbb{P}^1(k)\backslash \{\lambda_1,\ldots,\lambda_t\}$, $\mcal{U}_\lambda$ is a homogeneous tube and for $1\leq i\leq t$, $\mcal{U}_{\lambda_i}$ is a stable tube of rank $p_i$. We will refer to the sheaves in $\mcal{U}_\lambda$ for $\lambda\in\mathbb{P}^1(k)$ as being \textbf{sheaves concentrated at $\lambda$}.

There are linear forms $\text{rk}:K_0(\mathbb{X})\rightarrow \Z$ \cite[1.8.2]{GeiLen}, called \textbf{rank}, and $\text{deg}:K_0(\mathbb{X})\rightarrow \Z$ \cite[2.8]{GeiLen}, called \textbf{degree}. The linear form $\text{rk}$ is determined by $\text{rk}\mcal{O}(x)=1$ for all $x\in \mathbb{L}$ and the linear form $\text{deg}$ is determined by $\text{deg}\mcal{O}(x)=\delta(x)$ for all $x\in\mathbb{L}$. For all $X\in\coh(\mathbb{X})$, $\text{rk}X\geq 0$. A coherent sheaf $X$ has rank zero if and only if $X$ is finite length. If $X\in\coh(\mathbb{X})$ is finite length and non-zero then $\deg X>0$.

The vector bundles of rank $1$ are called \textbf{line bundles} and they are all isomorphic to vector bundles of the form $\mcal{O}(x)$ where $x\in\mathbb{L}$. Moreover, \cite[2.6]{GeiLen}, every vector bundle $F$ has a filtration by line bundles and the number of line bundles occurring in such a filtration is equal to the rank of $F$.


The \textbf{virtual genus} $g_\mathbb{X}:=1+\frac{1}{2}\delta(\omega)$ of $\mathbb{X}$ strongly controls the structure of $\coh(\mathbb{X})$. When $g_{\mathbb{X}}=1$ we say that $\mathbb{X}$ is \textbf{of tubular type}. Note that $g_\mathbb{X}$ is only dependent on the weight sequence of $\mathbb{X}$ and, up to permutation, the only weight sequences of tubular type are $(2,3,6)$, $(2,4,4)$, $(3,3,3)$ and $(2,2,2,2)$.

The Riemann-Roch formula, given here as in \cite[2.4]{LenMeltub}, relates the degree and rank of $x,y\in\K_0(\mathbb{X})$ with the Euler form:
\[\sum_{j=0}^{p-1}\langle \tau^jx,y\rangle= p(1-g_\mathbb{X})\text{rk}x\cdot\text{rk}y+\left|\begin{array}{cc}
                                             \text{rk}x & \text{rk}y \\
                                             \text{deg}x & \text{deg}y
                                           \end{array}\right|.
\]

The \textbf{(GL-)slope} of a non-zero vector bundle $X$ is given by $\mu(X):=\deg X/\text{rk} X$. If $X$ is a finite length sheaf then we set $\mu(X):=\infty$.

A vector bundle $X\in\coh(\mathbb{X})$ is \textbf{semistable} if for each $Y\subseteq X$, $\mu(Y)\leq \mu(X)$. For $q\in \Q\cup\{\infty\}$, let $\mcal{W}_q$ denote the full subcategory of all semistable sheaves of slope $q$.

When $\mathbb{X}$ is of tubular type the following theorem describes $\coh(\mathbb{X})$.

\begin{theorem}\label{cohtub}\cite[5.6]{GeiLen}\cite{LenMeltub}
Let $\mathbb{X}$ be a weighted projective line of tubular type.
\begin{enumerate}[(i)]
\item All indecomposable $X\in \coh(\mathbb{X})$ are semistable.  
\item For all $X\in\coh(\mathbb{X})$ indecomposable, $\mu(X)=\mu(\tau X)$.
\item For $X,Y\in \coh(\mathbb{X})$ indecomposable, if $\Hom(X,Y)\neq 0$ then $\mu(X)\leq \mu(Y)$.
\item For each $q\in\Q\cup\{\infty\}$, $\mcal{W}_q$ is equivalent to $\coh_0(\mathbb{X})$.
\end{enumerate}
\end{theorem}

\subsubsection{Concealed canonical algebras}
A vector bundle $\Sigma$ is said to be a \textbf{tilting bundle} if $\Ext^1(\Sigma,\Sigma)=0$ and $\text{D}^b(\mathbb{\X}):=\text{D}^b(\coh(\mathbb{X}))$ is the smallest triangulated subcategory of $\text{D}^b(\mathbb{X})$ containing $\Sigma$ (see \cite[3.1]{GeiLen}).
A \textbf{concealed canonical algebra} is the endomorphism ring of a tilting bundle in $\coh(\mathbb{X})$ for some weighted projective line $\mathbb{X}$. If $\mathbb{X}$ is of tubular type and $\Sigma\in\coh(\mathbb{X})$ is a tilting bundle then $\End(\Sigma)$ is a tubular algebra and all tubular algebras occur in this way \cite[3.6]{LenMeltilt}.

For any weighted projective line $\mathbb{X}(\mathbf{p},\mathbf{\lambda})$, Geigle and Lenzing defined a tilting bundle $\Sigma_{can}:=\oplus_{0\leq x\leq c}\mcal{O}(x)$, called the \textbf{canonical tilting bundle}. The endomorphism ring of $\Sigma_{can}$ is the canonical algebra $\Lambda(\mathbf{p},\mathbf{\lambda})$ in the sense of Ringel \cite[\S 4]{GeiLen}.

Suppose that $R$ is a concealed canonical algebra of tubular type, i.e. a tubular algebra, and $\Sigma\in\text{coh}(\mathbb{X})$ is a tilting vector bundle such that $\End(\Sigma)=R$.

Let $\mcal{T}$ be the \textbf{torsion class} of $\Sigma$, that is, the full subcategory of $\text{coh}(\mathbb{X})$ generated by $\Sigma$ (or equivalently, the full subcategory of objects $X\in\text{coh}(\mathbb{X})$ such that $\Ext(\Sigma,X)=0$) and let $\mcal{F}$ be the \textbf{torsion-free class} of $\Sigma$, that is, the full subcategory of objects $X\in\text{coh}(\mathbb{X})$ such that $\Hom(\Sigma,X)=0$.

Since $\text{coh}(\mathbb{X})$ is hereditary, the objects of the bounded derived category, $\text{D}^{\text{b}}(\mathbb{X})$, are of the form $\oplus_{i\in I}X_i[i]$ where $I\subseteq \Z$ is finite and $X_i\in\text{coh}(\mathbb{X})$ for all $i\in I$. For all $X,Y\in \text{coh}(\mathbb{X})$ and $i,j\in \Z$, \[\Hom_{D^b(\mathbb{X})}(X[i],Y[j])=\Ext^{j-i}_{\text{coh}(\mathbb{X})}(X,Y).\]

The right derived functor of $\Hom(\Sigma,-)$ gives equivalence of bounded derived categories
\[\text{RHom}(\Sigma,-):\D^\text{b}(\mathbb{X})\rightarrow \D^\text{b}(R):=\D^\text{b}(\mod\text{-}R)\] and $\mod\text{-}R$ is equivalent to the subcategory $\mcal{T}\vee\mcal{F}[1]$ of $\D^\text{b}(\mathbb{X})$ consisting of objects of the form $X\oplus Z[1]$ where $X\in \mcal{T}$ and $Z\in\mcal{F}$ (see \cite[\S 3]{GeiLen}). Moreover, $\text{RHom}(\Sigma,-)$ induces an Euler form preserving isomorphism of Grothendieck groups
\[K_0(\mathbb{X})\rightarrow K_0(R), \ \ \ [X]\mapsto \sum_{i=0}^\infty(-1)^i[\Ext_{\X}^{i}(\Sigma,X)].\]

In what follows, we will use this equivalence to identify $\D^\text{b}(\mathbb{X})$ and $\D^\text{b}(R)$ (and hence $K_0(\mathbb{X})$ and $K_0(R)$ equipped with their Euler forms).

We now recall, see \cite{LenMeltub} and \cite[4.9]{Kussinthesis}, how the various parts of $\mod\text{-}R$ sit in $\mcal{T}\vee\mcal{F}[1]$. This will allow us to link slope in the sense of Geigle and Lenzing and slope in the sense of Ringel.

Throughout this section, let $\mu_{\text{max}}$ (respectively $\mu_{\text{min}}$) be the maximal (respectively minimal) GL-slope of any indecomposable direct summand of $\Sigma$. Decompose the tilting bundle $\Sigma=\Sigma_0\oplus\Sigma_{\text{max}}=\Sigma_\infty\oplus \Sigma_{\text{min}}$ where $\Sigma_{\text{max}}$ (respectively $\Sigma_{\text{min}}$) is the sum of the indecomposable direct summands of $\Sigma$ with GL-slope $\mu_{\text{max}}$ (respectively $\mu_{\text{min}}$). Note that $\mu_{\text{min}}<\mu_{\text{max}}$ since if all indecomposable direct summands of $\Sigma$ had the same GL-slope then $\Sigma$ would not generate $\text{D}^b(\mathbb{X})$.

\begin{proposition}\label{modLambdaincoh}
\begin{enumerate}
\item If $X\in\coh(\mathbb{X})$ is indecomposable and $\mu(X)>\mu_{\text{max}}$ then $X\in\mcal{T}$.
\item If $Z\in\coh(\mathbb{X})$ is indecomposable and $\mu(Z)<\mu_{\text{min}}$ then $Z\in\mcal{F}$.
\item The indecomposable projective $R$-modules are the indecomposable direct summands of $\Sigma$ and the preprojective component of $R$ is equal to those $X\in\coh(\mathbb{X})$ with $\mu(X)<\mu_{\text{max}}$ and $\Ext_{\mathbb{X}}(\Sigma,X)=0$. So, in particular, the indecomposable direct summands of $\Sigma_{\text{max}}$ are exactly the projective $R$-modules of Ringel slope zero.
\item The indecomposable injective $R$-modules are $(\tau X)[1]$ where $X$ is an indecomposable direct summand of $\Sigma$ and the preinjective component of $R$ is equal to $Z[1]$ such that $Z\in \coh(\mathbb{X})$, $\mu(Z)>\mu_{\text{min}}$ and $\Hom(\Sigma,Z)=0$.
\end{enumerate}
\end{proposition}
\begin{proof}
(1) and (2) follow from \ref{cohtub} and Serre duality. For (3), see \cite[5.7]{LenMeltilt} and \cite[4.9]{Kussinthesis}. The first part of (4) is \cite[5.3]{LenMeltilt} and the rest follows from (3) using vector bundle duality as indicated in \cite[5.1]{LenMeltub}.
%
%
\end{proof}

Following \cite{LenMeltub}, let $u_j:=[\tau^j\mcal{O}]$ and $u:=\sum_{j=0}^{p-1}u_j$ and $w:=[S]$ where $S\in\coh(\X)$ is a simple sheaf concentrated at an ordinary point. Then $\langle u,x\rangle=\text{deg}x$ and $\langle w,x\rangle =-\text{rk} x$. By \cite[2.6]{LenMeltub}, $w$ and $u$ generate $\rad(K_0(\mathbb{X}))$ as an abelian group.

Let $h_0$ and $h_\infty$ be the canonical radical vectors of $R$ in the sense of Ringel. Let $\alpha_0,\alpha_\infty,\beta_0,\beta_\infty\in \Q$ be such that $h_0=\alpha_0 u+\beta_0 w$ and $h_{\infty}=\alpha_\infty u+\beta_\infty w$.

Suppose $X\in\mcal{T}$. Then
\[\langle h_0,[X]\rangle=\left\{
                          \begin{array}{ll}
                            \alpha_0\deg X, & \hbox{when $\text{rk} X=0$;} \\
                            \text{rk} X(\alpha_0\mu(X)-\beta_0), & \hbox{otherwise}
                          \end{array}
                        \right.
\] and
\[\langle h_\infty,[X]\rangle=\left\{
                          \begin{array}{ll}
                            \alpha_\infty\deg X, & \hbox{when $\text{rk} X=0$;} \\
                            \text{rk} X(\alpha_\infty\mu(X)-\beta_\infty), & \hbox{otherwise.}
                          \end{array}
                        \right.
\]

Suppose $Z\in\mcal{F}$. Then
\[\langle h_0,[Z[1]]\rangle=-\text{rk} Z(\alpha_0\mu(Z)-\beta_0)\] and
\[\langle h_\infty,[Z[1]]\rangle=-\text{rk} Z(\alpha_\infty\mu(Z)-\beta_\infty).\]

\begin{lemma}\label{calcGLtoLen}
With the notation as in the rest of this section, the following hold:
\begin{enumerate}[(i)]
\item $\beta_0=\mu_{\text{max}}\alpha_0$
\item $\beta_\infty=\mu_{\text{min}}\alpha_\infty$
\item $\alpha_0>0$ and $\alpha_{\infty}<0$.
\end{enumerate}
\end{lemma}
\begin{proof}
(i) If $X$ is an indecomposable direct summand of $\Sigma_{\text{max}}$ then $X\in\mcal{T}$ and $X$ is a projective $R$-module not in the preprojective component by \ref{modLambdaincoh} (3). Hence $X$ has Ringel slope $0$. Therefore $0=\langle h_0,[X]\rangle=\text{rk} X(\alpha_\infty\mu(X)-\beta_\infty)$. So, since $\text{rk}X>0$, $\beta_0=\mu_{\text{max}}\alpha_0$. Note that we can also conclude from this that $\alpha_0\neq 0$.

(ii) This is proved as (i) using \ref{modLambdaincoh} (4). As in (i), we can conclude that $\alpha_{\infty}\neq 0$.

(iii) If $X\in\coh(\mathbb{X})$ is indecomposable and $\mu(X)=\infty$ then, by \ref{modLambdaincoh}, as a $R$-module, $X$ is neither preprojective, of slope $0$ or preinjective. So, by \cite[5.2]{Ringeltub}, $\alpha_0\deg X>0$ and $\alpha_\infty \deg X\leq 0$. Therefore $\alpha_0>0$ and $\alpha_\infty<0$.
\end{proof}

\begin{definition}
Let $\gamma:\Q\cup\{\infty\}\rightarrow \Q\cup\{\infty\}$ be defined by
\[\gamma(q):= -\frac{\alpha_0}{\alpha_\infty}\cdot\frac{q-\mu_{\text{max}}}{q-\mu_{\text{min}}}\] for $q\in\Q\backslash\{\mu_{\text{min}}\}$, $\gamma(\mu_{\text{min}})=\infty$ and $\gamma(\infty)=-\alpha_0/\alpha_\infty$.
\end{definition}

The determinant of the m\"{o}bius transformation $\gamma$ is \[\alpha_0\beta_\infty-\alpha_\infty\beta_0=\alpha_0\alpha_\infty(\mu_{\text{min}}-\mu_{\text{max}}).\] So, by \ref{calcGLtoLen}, the determinant is strictly positive. It is now easy to see that \[\gamma|_{(-\infty,\mu_{\text{min}})}:(-\infty,\mu_{\text{min}})\rightarrow (-\alpha_0/\alpha_\infty,\infty)\] and
\[\gamma|_{(\mu_{\text{max}},\infty]}:(\mu_{\text{max}},\infty]\rightarrow (0,-\alpha_0/\alpha_\infty]\] are both strictly increasing and bijective.

So, in particular, if $X\in\coh(\mathbb{X})$ is indecomposable and $q:=\mu(X)>\mu_{\text{max}}$ then the Ringel slope of $X$  is $\gamma(q)\in (0,-\alpha_0/\alpha_\infty]$ and if $Z\in \coh(\mathbb{X})$ is indecomposable and $q:=\mu(Z)<\mu_{\text{min}}$ then the Ringel slope of $Z[1]$ is $\gamma(q)\in (-\alpha_0/\alpha_\infty,\infty)$.

%
%
%

\subsubsection{Accumulation points for coherent sheaves and tubular algebras}

\begin{remark}\label{coprime}
Let $\mathbb{X}$ be a weighted projective line of tubular type. Let $(a,b)\in \Z\times \N\cup \{(1,0)\}$ with $a,b$ coprime. Every quasi-simple $Y\in\text{coh}(\mathbb{X})$ in a tube of rank $p:=\text{lcm}\{p_1,\ldots,p_t\}$ with $\mu(Y)=a/b$ has $\text{rk}Y=b$ and $\text{deg}Y=a$.
\end{remark}
\begin{proof}
Let $E$ be a simple sheaf concentrated at an exceptional point of weight $p$. By definition, see \cite[2.8]{GeiLen}, $\text{deg}E=1$ and $\text{rk}E=0$. For all $q\in \Q$ there is an equivalence, defined in \cite{LenMeltub}, called a telescopy functor, $\Phi_{q,\infty}:\mcal{C}_\infty\rightarrow \mcal{C}_q$. These functors are compositions of shift functors $S:\text{coh}(\mathbb{X})\rightarrow \text{coh}(\mathbb{X})$ given on objects by $SX=X(x_i)$ where $1\leq i\leq t$ is such that $p_i=p$, inverses of shift functors and right mutations $R_q:\mcal{C}_q\rightarrow \mcal{C}_{q/1+q}$ for $0<q\leq \infty$.  If $X\in\text{coh}(\mathbb{X})$ then $\text{rk}SX=\text{rk}X$ and $\text{deg}SX=\text{deg}X+\text{rk}X$. If $X\in\text{coh}(\mathbb{X})$ and $0<\mu(X)\leq \infty$ then $\text{rk}RX=\text{rk}X+\text{deg}X$ and $\text{deg}RX=\text{deg}X$. So $S$ and $R$, and hence $\Phi_{q,\infty}$ preserve coprimeness of rank and degree.
\end{proof}

\begin{proposition}\label{acccoh}
Let $\mathbb{X}$ be a weighted projective line of tubular type.
\begin{enumerate}[(a)]
\item Let $q\in \mathbb{Q}$ and $Y$ be a quasi-simple in a tube of rank $p$ with $\mu(Y)=q$. There exist
\begin{itemize}
\item $(q_n)_{n\in\N}$ a strictly decreasing sequence with $q_n\in \Q$ such that $q_n\rightarrow q$ as $n\rightarrow \infty$, and
\item $X_n,Z_n\in \coh(\mathbb{X})$ with $\mu(X_n)=\mu(Z_n)=q_n$, $\Hom(Y,X_n)=0$ and $\Hom(Y,Z_n)\neq 0$.
\end{itemize}

\item There exist
\begin{itemize}
\item $Y\in\text{coh}(\mathbb{X})$ with $\mu(Y)=\infty$, and
\item $X_n,Z_n\in\text{coh}(\mathbb{X})$ with $\mu(X_n)=\mu(Z_n)=-n$, $\Hom(Z_n,Y)=0$ and $\Hom(Z_n,Y)\neq 0$.
\end{itemize}
\end{enumerate}
\end{proposition}
\begin{proof}
(a) Let $q\in \mathbb{Q}$. By \ref{coprime}, $r:=\text{rk}Y>0$ and $d:=\text{deg}Y$ are coprime and $d/r=q$. Since $r$ and $d$ are coprime, there exists $a_0\in\Z$ and $b_0\in\N$ such that $ra_0-b_0d=1$.  For all $n\in\N$, let $a_n=a_0+nd$ and $b_n=b_0+nr$. Then $ra_n-b_nd=1$ for all $n\in\N_0$ and hence $a_n$ and $b_n$ are coprime. Moreover $q_n:=a_n/b_n$ is a strictly decreasing sequence of rational numbers such that $q_n\rightarrow d/r=q$ as $n\rightarrow \infty$.

Since $a_n$ and $b_n$ are coprime, by \ref{coprime}, for each $n\in\N$, there exists a quasi-simple $W$ in a tube of rank $p$ such that $\text{rk}(W)=b_n$ and $\text{deg}(W)=a_n$. By the Riemann-Roch equation and since $q_n>q$,
\[\sum_{j=0}^{p-1}\dim\Hom(Y,\tau^jW)=ra_n-b_nd=1.\] Therefore $\dim\Hom(Y,\tau^jW)\neq 0$ for exactly one $0\leq j\leq p-1$.

(b) The argument is similar to part (a) and left to the reader.
\end{proof}

\begin{proof}[Proof of \ref{accallslopes}.]
Let $\mathbb{X}$ be a weighted projective line and $\Sigma\in\text{coh}(\mathbb{X})$ a tilting bundle such that $\End(\Sigma)\cong R$. We keep the notation as in the rest of this section.

First suppose that $q'\in (0,-\alpha_0/\alpha_{\infty})$. Let $q\in (\mu_{\text{max}},\infty)$ be such that $\gamma(q)=q'$. Let $Y,X_n,Z_n\in\text{coh}(\mathbb{X})$ and $q_n\in \Q$ be as in \ref{acccoh}(a). Since $\mu(Y),\mu(X_n),\mu(Z_n)>\mu_{\text{max}}$, $Y, X_n,Z_n\in \mcal{T}$. So $\Hom_R(Y,X_n)=0$ for all $n\in \N$ and $\Hom_R(Y,Z_n)\neq 0$ for all $n\in\N$. Let $q_n':=\gamma(q_n)$. Then $\text{slope}Y=q'$, $\text{slope} X_n=\text{slope} Z_n=q_n'$ and $q_n'\rightarrow q$ as $n\rightarrow \infty$. So $q'$ is an accumulation point of the set of non-uniform slopes for $\Hom_R(Y,-)$.

The case when $q'\in (-\alpha_0/\alpha_{\infty},\infty)$ is similar and left to the reader.

Suppose that $q'=-\alpha_0/\alpha_{\infty}$. Let $Y,X_n,Z_n\in\text{coh}(\mathbb{X})$ be as in \ref{acccoh}(b). Then $\mu(\tau X_n)=\mu(\tau Z_n)=-n$ for all $n\in\N$.  For all $n\geq -\mu_{\text{min}}+1$, $X_n,Z_n\in\mcal{F}$. By Serre duality, $0=D\Hom_\mathbb{X}(X_n,Y)=\Ext_{\mathbb{X}}(Y,\tau X_n)$ and $0\neq D\Hom_\mathbb{X}(Z_n,Y)=\Ext_{\mathbb{X}}(Y,\tau Z_n)$. So $\Hom_R(Y,(\tau X_n)[1])=0$ and $\Hom_R(Y,(\tau X_n)[1])\neq 0$ for all $n\geq -\mu_{\text{min}}+1$. It just remains to note that the Ringel slope of $X_n$ and $Z_n$, that is $\gamma(-n)$, tends to $-\alpha_0/\alpha_{\infty}$ as $n\rightarrow \infty$.

Suppose $q'=0$. The description of the tilting objects of $\text{coh}(\mathbb{X})$ given in \cite[3.1 \& 3.5]{LenMeltilt} means that if $\mathbb{X}$ is of tubular type and if $\mcal{T}$ is inhomogeneous tube of slope $\mu_{\text{max}}$ then if $\Sigma_{\text{max}}$ has a direct summand in $\mcal{T}$ then $\Sigma_{\text{max}}$ has a quasi-simple from $\mcal{T}$ as a direct summand. Let $Y$ be a quasi-simple of slope $\mu_{\text{max}}$ in a tube $\mcal{T}$ of rank $p$. If $\Sigma_{\text{max}}$ has a direct summand from $\mcal{T}$ then assume that $Y$ is a direct summand of $\Sigma_{\text{max}}$. In either case, $\Ext(\Sigma_{\text{max}},Y)=0$ and hence $\Ext(\Sigma,Y)=0$.

Now, arguments as in \ref{acccoh} imply that there exist a strictly decreasing sequence $q_n\in \Q$ such that $q_n\rightarrow q$ as $n\rightarrow \infty$ and $X_n,Z_n\in\text{coh}(\mathbb{X})$ indecomposable of slope $q_n$ such that $\Hom_{\mathbb{X}}(Y,X_n)=0$ and $\Hom_{\mathbb{X}}(Y,Z_n)\neq 0$. Since $\mu(X_n)=\mu(Z_n)>\mu_{\text{max}}$ for $n\in\N$, $Y,X_n,Z_n\in \mcal{T}$. The argument now proceeds as in the previous cases. \end{proof}

\section{Almost all slopes and Presburger arithmetic}\label{almostallandpresburger}

The language of Presburger arithmetic is $\mcal{L}_{\text{Pr}}:=(+,<,0)$ where $+$ is a binary function symbol, $<$ is a binary relation symbol and $0$ is a constant symbol. Presburger arithmetic is the theory of $\Z$ in $\mcal{L}_{\text{Pr}}$ where $+$ is interpreted as the usual addition on $\Z$, $<$ is interpreted as the usual order on $\Z$ and $0$ is interpreted as the additive unit in $\Z$. Presburger arithmetic is decidable. For more information about Presburger arithmetic see \cite{Markermodeltheory} (see \cite[3.1.21]{Markermodeltheory} for the proof of decidability).

We start this section by showing that for a tubular algebra $R$, the set of $x\in \Z^n\cong K_0(R)$ such that $x$ is the dimension vector of some indecomposable $X\in\mod\text{-}R$ is a definable subset of $\Z^n$ in the language of Presburger arithmetic \ref{dimvecofinddefinPres}. In order to do this, we will use the fact that $\mod\text{-}R$ is controlled by $\chi_R$, in particular that the dimension vectors of indecomposable finite-dimensional $R$-modules correspond exactly to the positive connected radical and root vectors of $\chi_R$. Note, however, that if we add a function symbol $\chi$ to Presburger arithmetic and interpret it as any non-zero quadratic form on $\Z$ then we can define multiplication in $\Z$ and hence, the theory becomes undecidable. So instead we argue that for $\chi_R$ the Euler quadratic form on $K_0(R)$, the set of $x\in\Z^n$ such that $\chi_R(x)=0$ or $\chi_R(x)=1$ is already a definable subset of $\Z^n$ in the language of Presburger arithmetic.

\begin{lemma}\label{puregroupdef} \quad For any pure subgroup $G$ of $\Z^n$ there is an $n$-formula $\Delta(x_1,\ldots,x_n)$ in the language of Presburger arithmetic such that $(g_1,\ldots,g_n)$ is in $G$ if and only if $\Delta(g_1,\ldots,g_n)$ holds in $\Z$.
\end{lemma}
\begin{proof}
Let $V$ be the $\Q$-linear span of $G$ as a subset of $\Q^n$. Since $G$ is pure, $V\cap \Z^n=G$. Since $V$ is a subspace of $\Q^n$ there is a matrix $A$ with entries from $\Q$ such that $v\in V$ if and only if $vA=0$. By multiplying $A$ by some integer, we may assume that $A$ has integer entries. Now, for any $g\in \Z^n$, $g\in G$ if and only if $gA=0$. Let $\Delta(x_1,\ldots,x_n)$ be the formula $(x_1,\ldots,x_n)A=0$. Note that $\Delta$ is a formula without parameters.
\end{proof}

\begin{cor}
Let $R$ be a tubular algebra. The group $\rad\chi_R\subseteq \Z^n\cong K_0(R)$ is definable in the language of Presburger arithmetic.
\end{cor}
\begin{proof}
Recall, that, when $R$ is a tubular algebra, $\chi_R$ is positive semi-definite and hence $\rad\chi_R$ is a subgroup of $K_0(R)$. If $x\in K_0(R)$ and $nx\in \rad\chi_R$ for some $n\in\Z\backslash\{0\}$ then $n^2\chi_R(x)=\chi_R(nx)=0$ and hence $x\in\rad\chi_R$. So $\rad\chi_R$ is pure in $K_0(R)$.
\end{proof}

If $R$ is a tubular algebra, $\chi_R(x)=1$ and $x-y\in \rad\chi_R$ then, since $\chi_R$ is positive semi-definite, $\chi_R(y)=\chi_R(x-(x-y))=\chi_R(x)=1$.

Similar results to the following have been obtained purely K-theoretically in \cite[2.3]{KussinKtheory}. However we require exactly the formulation of \ref{Omegaexists}.

\begin{lemma}\label{Omegaexists}
Let $R$ be a tubular algebra. There is a finite subset $\Omega\subseteq K_0(R)$ such that for all $x\in K_0(R)$ with $\chi_R(x)=1$, there exists $y\in\Omega$ such that $x-y\in\rad\chi_R$.
\end{lemma}
\begin{proof}
Suppose that no such finite set $\Omega\subseteq K_0(R)$ exists. Then there are infinitely many $y$ with $\chi_R(y)=1$ all in pairwise distinct cosets of $\rad \chi_R$. Note that if $\lambda,\mu\in\Z$ then $\langle h_0,y+\lambda h_0+\mu h_\infty\rangle=\langle h_0,y\rangle+\mu\langle h_0,h_\infty\rangle$ and $\langle h_\infty,y+\lambda h_0+\mu h_\infty\rangle=\langle h_\infty, y\rangle+\lambda\langle h_\infty,h_0\rangle$. Let $a,b\in\N$ be such that $a=\langle h_0,h_\infty\rangle$ and $-b=\langle h_\infty,h_0\rangle$. Thus, there are infinitely many $y$ with $\chi_R(y)=1$ in pairwise different cosets of $\rad\chi_R$ such that $0< \langle h_0,y\rangle\leq a$ and $-b\leq\langle h_\infty,y\rangle< 0$. Therefore, there exists $e,f\in \N$ such that there are infinitely many $y$ with $\chi_R(y)=1$ in pairwise different cosets of $\rad\chi_R$ such that $\langle h_0,y\rangle=e$ and $\langle h_\infty,y\rangle=-f$.

Let $x=fh_0+eh_\infty$. Note that by \cite[5.1.1]{Ringeltub} $x$ is sincere i.e. $x_i>0$ for all $1\leq i\leq n$ where $x=(x_1,\ldots,x_n)$. A quick calculation gives that $-e\langle h_\infty,x\rangle=f\langle h_0,x\rangle$. Since $x$ is sincere, for any $y$ in our infinite set, there is a $c\in \N_0$ such that $y+cx$ is positive and connected; note that $-e\langle h_\infty,y+cx\rangle=f\langle h_0,y+cx\rangle$. Thus we have an infinite set of elements $z\in K_0(R)$ such that $z$ is connected, positive, $\chi_R(z)=1$ and $-e\langle h_\infty,z\rangle=f\langle h_0,z\rangle$ all of which are pairwise in different cosets of $\rad \chi_R$. This contradicts that fact that for each slope $q$, $R$ has only finitely many inhomogeneous tubes.
\end{proof}

\begin{lemma}\label{dimvecofinddefinPres}
Let $R$ be a tubular algebra. The set of dimension vectors $x\in K_0(R)$ such that $\chi_R(x)=0$ or $\chi_R(x)=1$ is definable in the language of Presburger arithmetic. Thus, the set of dimension vectors of finite-dimensional indecomposable modules over $R$ is definable in the language of Presburger arithmetic.
\end{lemma}
\begin{proof}
By lemma \ref{Omegaexists}, there is a finite subset $\Omega\subseteq K_0(R)$ such that for all $x\in K_0(R)$ with $\chi_R(x)=1$, there exists $y\in\Omega$ such that $x-y\in\rad\chi_R$. By \ref{puregroupdef} $\rad\chi_R$ is definable in the language of Presburger arithmetic. Thus, the set of dimension vectors $x\in K_0(R)$ such that $\chi_R(x)=0$ or $\chi_R(x)=1$ is definable in the language of Presburger arithmetic.

That $x=(x_1,\ldots,x_n)\in K_0(R)$ is positive is expressed by saying $x_i\geq 0$ for all $1\leq i\leq n$ and that $x_i>0$ for some $1\leq i\leq n$. That $x=(x_1,\ldots,x_n)\in K_0(R)$ is connected is expressed by saying that if $x_i>0$ and $x_j>0$ then there is some path $P$ in the underlying quiver of $R$ between $i$ and $j$ such that for all vertices $k$ in $P$, $x_k>0$.

By \ref{Ringelcontrol}, $\mod\text{-}R$ is controlled by $\chi_R$. Thus any connected positive dimension vector with $\chi_R(x)=0$ or $\chi_R(x)=1$ is the dimension vector of an indecomposable module and all dimension vectors of indecomposable modules are of this form. Thus we have shown that the set of dimension vectors of finite-dimensional indecomposable modules over $R$ is definable in the language of Presburger arithmetic.
\end{proof}

\begin{proposition}\label{nonexpenoughtocheckfd}
Let $R$ be a tubular algebra and $\phi/\psi,\phi_1/\psi_1,\ldots,\phi_n/\psi_n$ be pp-pairs such that there exist $v,w_1,\ldots,w_n\in K_0(R)$ such that for indecomposable finite-dimensional modules $M$ with slope in the interval $(a,b)$,
\[\dim\phi(M)/\psi(M)=v\cdot\udim M\] and for $1\leq i\leq n$
\[\dim \phi_i(M)/\psi_i(M)=w_i\cdot\udim M.\]

If there is an indecomposable pure-injective module $N$ with slope in $(a,b)$ such that $N\in(\phi/\psi)$ but $N\notin \bigcup_{i=1}^n(\phi_i/\psi_i)$ then there is a finite-dimensional indecomposable module $M$ with slope in $(a,b)$ such that $M\in(\phi/\psi)$ but $M\notin \bigcup_{i=1}^n(\phi_i/\psi_i)$.
\end{proposition}
\begin{proof}
Suppose that $N$ is as in the statement and that $N$ has slope $q$.

For any slope $p\in(a,b)$ either $\phi/\psi$ is closed on all modules of slope $p$ or $\phi/\psi$ is open on all the indecomposable pure-injective modules of slope $p$ except for finitely many finite-dimensional indecomposable modules. See \ref{closedhomqimpliesclosedq} and \ref{openonhomoimpliesoenonallbutfinitelymany} for $p$ rational and \ref{onepointuptoelemeqatirrational} for $p$ irrational.

So, $N\in(\phi/\psi)$  implies that $\phi/\psi$ is open on almost all indecomposable pure-injectives of slope $q$ and $N\notin (\phi_i/\psi_i)$ implies $\phi_i/\psi_i$ is closed on all indecomposable pure-injectives of slope $q$. So if $q$ is rational then there is a finite-dimensional indecomposable module $M$ such that $M\in(\phi/\psi)$ and $M\notin(\phi_i/\psi_i)$ for $1\leq i\leq n$.

If $q$ is irrational then there is some $\epsilon>0$ such that $\phi/\psi$ is open on all finite-dimensional indecomposable $M$ with slope in $(q-\epsilon,q+\epsilon)$ \cite[8.7]{modirrslope}. Likewise, for each $1\leq i\leq n$, there is some $\epsilon_i>0$ such that $\phi_i/\psi_i$ is closed on all finite-dimensional indecomposable $M$ with slope in $(q-\epsilon_i,q+\epsilon_i)$. This is true because if $\mcal{C}_q\subseteq \Zg_R\backslash \left(\phi_i/\psi_i\right)$ then $\left(\phi_i/\psi_i\right)\subseteq \Zg_R\backslash\mcal{C}_q=\bigcup_{\epsilon>0}\Zg_R\backslash \mcal{C}_{(q-\epsilon,q+\epsilon)}$. Since $\left(\phi_i/\psi_i\right)$ is compact, there exists some $\epsilon>0$ such that $\left(\phi_i/\psi_i\right)\subseteq \Zg_R\backslash\mcal{C}_{(q-\epsilon,q+\epsilon)}$.

Thus there is some finite-dimensional indecomposable module $M$ with slope in $(a,b)$ such that $M\in\left(\phi/\psi\right)$ and $M\notin\bigcup_{i=1}^n\left(\phi_i/\psi_i\right)$.
\end{proof}

\begin{lemma}\label{algorfornonexp}
There is an algorithm which given $w,v_1,\ldots,v_n\in\Z^m$ and $a<b\in \Q_0^\infty$ answers whether there is an indecomposable finite-dimensional module $X$ with slope in $(a,b)$ such that $w\cdot \udim X>0$ and for $1\leq i\leq n$, $v_i\cdot \udim X=0$.
\end{lemma}
\begin{proof}
Note that there are vectors $g_0$ and $g_{\infty}$ such that for all $x\in\Z^m$, $\langle h_0,x\rangle =g_0\cdot x$ and $\langle h_\infty,x\rangle =g_\infty\cdot x$.

Thus $x\in\Z^m$ has ``slope'' in $(a,b)$ if and only if $-(g_0\cdot x)/(g_\infty\cdot x)\in (a,b)$. This statement can be easily rewritten in the language of Presburger arithmetic.

In \ref{dimvecofinddefinPres}, we showed that set of dimension vectors of indecomposable finite-dimensional modules over $R$ is definable in Presburger arithmetic. Thus, since Presburger arithmetic is decidable, there is an algorithm which decides whether there is an $x\in\Z^m$ such that $x$ is the dimension vector of an indecomposable finite-dimensional module over $R$, $x$ has slope in $(a,b)$, $w\cdot x>0$ and for $1\leq i\leq n$, $v_i\cdot x=0$.
\end{proof}

\section{Decidability for theories of modules over tubular algebras}\label{maintheorem}

In this section we combine the results of the previous sections in order to prove that if $R$ is a tubular algebra over a recursive algebraically closed field then $R$ has decidable theory of modules.

\begin{theorem}\label{mainthm}
Let $R$ be a tubular algebra over a recursive algebraically closed field. The common theory of $R$-modules is decidable.
\end{theorem}

\begin{proof}[Proof of 9.1 for canonical algebras of tubular type]
It is enough to show that there is an algorithm which, given pp-pairs $\phi/\psi,\phi_1/\psi_1,\ldots,\phi_n/\psi_n$, answers whether
\[\left(\phi/\psi\right)\subseteq\bigcup_{i=1}^n\left(\phi_i/\psi_i\right).\]

First we show that there is an algorithm which answers whether there is an indecomposable pure-injective $N$ of strictly positive non-infinite slope with $N\in(\phi/\psi)$ such that $N\notin \bigcup_{i=1}^n(\phi_i/\psi_i)$.

By \ref{dimvectorsforpppairs}, there is an algorithm which, given $\phi/\psi,\phi_1/\psi_1,\ldots,\phi_n/\psi_n$, outputs $0=q_0<q_1<\ldots<q_m<q_{m+1}=\infty$ and $v_j$,$w_{ij}$ such that for all $0\leq j \leq m$ and all indecomposable finite-dimensional modules $N$ with slope in $(q_j,q_{j+1})$,

\[\dim \phi(N)/\psi(N)=v_j\cdot \udim N\]

and

\[\dim \phi_i(N)/\psi_i(N)=w_{ij}\cdot \udim N .\]

By \ref{nonexpenoughtocheckfd}, if there is an indecomposable pure-injective module $N$ with slope in $(q_j,q_{j+1})$ such that $N\in(\phi/\psi)$ and $N\notin \bigcup_{i=1}^n(\phi_i/\psi_i)$ then there is a finite-dimensional indecomposable module with slope in $(q_j,q_{j+1})$ such that $N\in(\phi/\psi)$ and $N\notin \bigcup_{i=1}^n(\phi_i/\psi_i)$. Thus, by \ref{algorfornonexp}, we can effectively answer whether there is an indecomposable pure-injective module $N$ with slope in $(q_j,q_{j+1})$ such that $N\in(\phi/\psi)$ and $N\notin \bigcup_{i=1}^n(\phi_i/\psi_i)$.

By \ref{algorforparticularq}, there is an algorithm which, for each $1\leq j\leq m$ answer whether
\[\left(\phi/\psi\right)\cap\mcal{C}_{q_j}\subseteq\bigcup_{i=1}^n\left(\phi_i/\psi_i\right)\cap\mcal{C}_{q_j} .\]

It now remains to answer whether there is an indecomposable pure-injective module $N\in \mcal{P}_0\cup\mcal{C}_0$ or $N\in \mcal{C}_{\infty}\cup\mcal{Q}_{\infty}$ such that $N\in \left(\phi/\psi\right)$ and $N\notin \bigcup_{i=1}^n\left(\phi_i/\psi_i\right)$.

Since $\mcal{P}_0\cup\mcal{C}_0\subseteq \mcal{E}_0\cup\mcal{E}_1$ and $\mcal{C}_{\infty}\cup\mcal{Q}_{\infty}\subseteq \mcal{E}'_0\cup \mcal{E}'_1$, it is enough to check if there is an indecomposable pure-injective module $N\in\mcal{E}_0\cup\mcal{E}_1 $ or $N\in \mcal{E}'_0\cup \mcal{E}'_1$ such that $N\in \left(\phi/\psi\right)$ and $N\notin \bigcup_{i=1}^n\left(\phi_i/\psi_i\right)$.  For this we refer to \ref{algforslopezero} and \ref{algforslopeinfinity}.
\end{proof}

We now extend the above result to tubular algebras. Note that, since the results of sections \ref{algslopeq}, \ref{corharlandprest} and \ref{almostallandpresburger} are for general tubular algebras, the only part of the proof missing is an algorithm which given pp-pairs $\phi/\psi,\phi_1/\psi_1,\ldots,\phi_n/\psi_n$ answers yes or no such that
\begin{enumerate}
\item if the algorithm answers yes then there is an (indecomposable pure-injective) $R$-module $N$ such that $N\in \left(\phi/\psi\right)$ and $N\notin \bigcup_{i=1}^n\left(\phi_i/\psi_i\right)$ and
\item if the algorithm answers no then there does not exist  $N\in \mcal{H}$ such that $N\in \left(\phi/\psi\right)$ and $N\notin \bigcup_{i=1}^n\left(\phi_i/\psi_i\right)$ where $\mcal{H}:=\mcal{P}_{0}\cup\mcal{C}_{0}\cup\mcal{C}_{\infty}\cup\mcal{Q}_{\infty}$.
\end{enumerate}

Using Herzog's duality, as in \ref{Herdual0toinf}, it is sufficient to replace $\mcal{H}$ in the above by $\mcal{P}_{0}\cup\mcal{C}_{0}$.

Let $\Gamma$ be a finite-dimensional algebra. A finite-dimensional $\Gamma$-module $T$ is a tilting module if the following three conditions are satisfied:
\begin{itemize}
\item [(T$1$)] $T$ has projective dimension less than or equal to $1$,
\item [(T$2$)] $\Ext^1(T,T)=0$, and
\item [(T$3$)] There exists a short exact sequence $\xymatrix@C=0.35cm{
  0 \ar[r] & \Gamma \ar[rr] && T' \ar[rr] && T'' \ar[r] & 0 }$ where $T'$ and $T''$ are direct summands of some finite power of $T$.
\end{itemize}

Note that, \cite[2.1]{Bongartz}, (T$3$) can be replaced with the condition that the number of pairwise non-isomorphic indecomposable direct summands of $T$ is equal to the number of pairwise non-isomorphic simple $\Gamma$-modules.

For an introduction to tilting theory for finite-dimensional algebras, see \cite[Chapter VI]{Ass1}.

\begin{proposition}\label{tiltfromcan}
Let $R$ be a tubular algebra. There exists a canonical algebra $\Gamma$ of tubular type and a tilting module $\Sigma\in\mod\text{-}\Gamma$ with $\End(\Sigma)\cong R$ such that for all indecomposable pure-injective $R$-modules $N$ which are either of slope zero or preprojective, there exists an indecomposable pure-injective $\Gamma$-module $M$ with $\Hom_\Gamma({_R}\Sigma,M)\cong N$.
\end{proposition}
\begin{proof}
Let $\Sigma$ be a tilting vector bundle in $\coh(\mathbb{X})$ where $\mathbb{X}$ is of tubular type such that $\End(\Sigma)\cong R$. Let $\Gamma$ be the endomorphism ring of the canonical tilting bundle $\Sigma_{\text{can}}:=\bigoplus_{0\leq x\leq c}\mcal{O}(x)$. Then $\Gamma$ is a canonical algebra of tubular type. We will view the categories $\coh(\mathbb{X})$, $\mod\text{-}\Gamma$ and $\mod\text{-}R$ as subcategories of $\D^{b}(\mathbb{X})$.

Let $\coh_{\geq}(\mathbb{X})$ be the torsion class of $\Sigma_{\text{can}}$ and $\coh_{-}(\mathbb{X})$ be the torsion-free class of $\Sigma_{\text{can}}$. So $\mod\text{-}\Gamma$ is equivalent to $\coh_{\geq}(\mathbb{X})\vee\coh_{-}(\mathbb{X})[1]$. Note that the maximal (respectively minimal) GL-slope of any $\mcal{O}(x)$ for $0\leq x\leq c$ is $p=\text{lcm}\{p_1,\ldots,p_t\}$ (respectively $0$).

By repeatedly applying the shift functor, which acts on sheaves by sending $X$ to $X(x_t)$, we may assume that each indecomposable direct summand of $\Sigma$ has slope strictly greater than $p$ and hence $\Sigma\in \coh_{\geq}(\mathbb{X})$. Since $\Sigma$ generates $\D^{b}(\mathbb{X})\cong\D^{b}(\Gamma)$, the indecomposable direct summands of $\Sigma$ generate $K_0(\Gamma)$. So, viewed as a $\Gamma$-module $\Sigma$ satisfies (T$3$). That $\Sigma$ has projective dimension less that or equal to $1$ follows from \cite[3.1.5]{Ringeltub}. Therefore $\Sigma\in\mod\text{-}\Gamma$ is a tilting module.

Let $\mcal{T}$ be the torsion class induced on $\coh(\mathbb{X})$ by $\Sigma$. Let $\mu_{\text{max}}$ be the maximal slope of any indecomposable direct summand of $\Sigma$ and let $\mu_{\text{min}}$ be the minimal slope of any indecomposable direct summand of $\Sigma$. Note that if $X\in\mcal{T}$ is indecomposable then $\mu_{\text{min}}\leq \mu(X)$. So, in particular, $p<\mu(X)$ and hence $X\in \coh_{\geq}(\mathbb{X})$. Moreover, if $X\in\mcal{T}$ then $\Ext^1_\Gamma(\Sigma,X)=\Ext^1_\mathbb{X}(\Sigma,X)=0$ and $\Hom_\Gamma(\Sigma,X)\cong\Hom_\mathbb{X}(\Sigma,X)$ as $R$-modules.

Since $\mcal{T}$ contains the preprojective component and all finite-dimensional $R$-modules of slope $0$, the image of the torsion class of $\Sigma$ in $\mod\text{-}\Gamma$ under $\Hom_{\Gamma}(\Sigma,-)$ in $\Mod\text{-}R$ contains all preprojective $R$-modules and all finite dimensional $R$-modules of slope zero.

Since $\Sigma$ is a tilting module in $\mod\text{-}\Gamma$, $\Hom_\Gamma(\Sigma,-):\Mod\text{-}\Gamma\rightarrow\Mod\text{-}R$ induces an equivalence between the torsion class $\mcal{G}$ in $\Mod\text{-}\Gamma$ of $\Gamma$-modules $M$ with $\text{Ext}_\Gamma(\Sigma,M)=0$ and the torsion-free class $\mcal{Y}$ in $\Mod\text{-}R$ of $R$-modules $N$ with $\text{Tor}_R(N,\Sigma)=0$ by \cite[3.5.1]{ColbyFuller}. By \cite[10.2.36]{PSL}, $\mcal{Y}$ is a definable subcategory of $\Mod\text{-}R$. Since, by \ref{fddenseinslopezero}, the smallest definable subcategory of $\Mod\text{-}R$ containing all finite-dimensional $R$-modules of slope $0$ contains all $R$-modules of slope $0$ and $\mcal{Y}$ contains all finite-dimensional $R$-modules of slope $0$, it follows that all $R$-modules of slope $0$ are in the image of $\Hom_\Gamma(\Sigma,-):\Mod\text{-}\Gamma\rightarrow \Mod\text{-}R$.
\end{proof}

\begin{proof}[proof of 9.1 for tubular algebras via tilting]
Let $R$ be a tubular algebra. Let $\Gamma$ and $\Sigma\in\mod\text{-}\Gamma$ be as in \ref{tiltfromcan}.

Since $F:=\Hom_\Gamma(\Sigma,-)$ is a $k$-linear interpretation functor, given a pp-pair $\phi/\psi$ over $R$, we can effectively construct a pp-pair $\phi'/\psi'$ over $\Gamma$ such that for all $M\in \Mod\text{-}\Gamma$, $|\phi'(M)/\psi'(M)|>1$ if and only if $|\phi(FM)/\psi(FM)|>1$. By the previous discussion and since we have already shown that the theory of $\Gamma$-modules is decidable, this is enough.
\end{proof}

\begin{cor}\label{concandec}
Prest's conjecture is true for concealed canonical algebras.
\end{cor}
\begin{proof}
As a consequence of \ref{mainthm}, it remains to confirm that domestic concealed canonical algebras have decidable theory of modules and that wild concealed canonical algebras have undecidable theory of modules.

By \cite[5.7]{LenMeltilt}, if $\Lambda$ is a wild concealed canonical algebra then $\Lambda$ is strictly wild and hence, by \cite{Epiintandstrwild}, has undecidable theory of modules.

By \cite[7.1]{LenPensep}, if $\Lambda$ is a domestic concealed canonical algebra then $\Lambda$ is tame concealed. So, \cite[17.17]{PreBk}, $\Lambda$ has decidable theory of modules.
\end{proof}

\textbf{Acknowledgements}
\noindent
I would like to thank Mike Prest for many helpful conversations. I would also like to thank Lidia Angeleri-H\"ugel for her showing me a draft of \cite{AngeleriHugelKussin} and for a very useful conversation about some of the material which is now in section \ref{1pointext}. I thank the anonymous referee for helpful comments which have improved this paper, in particular for sketching a proof of Proposition \ref{accallslopes}.
\bibliographystyle{alpha}
\bibliography{tubularalgebras}

\newcommand{\etalchar}[1]{$^{#1}$}
\begin{thebibliography}{EGN{\etalchar{+}}98}

\bibitem[AHK17]{AngeleriHugelKussin}
L.~Angeleri~H\"ugel and D.~Kussin.
\newblock Tilting and cotilting modules over concealed canonical algebras.
\newblock {\em Math. Z.}, 285(3-4):821--850, 2017.

\bibitem[ASS06]{Ass1}
I.~Assem, D.~Simson, and A.~Skowro{\'n}ski.
\newblock {\em Elements of the representation theory of associative algebras.
  {V}ol. 1}, volume~65 of {\em London Mathematical Society Student Texts}.
\newblock Cambridge University Press, Cambridge, 2006.
\newblock Techniques of representation theory.

\bibitem[Bau75]{decundectheoriesabgroups}
W.~Baur.
\newblock Decidability and undecidability of theories of abelian groups with
  predicates for subgroups.
\newblock {\em Compositio Math.}, 31(1):23--30, 1975.

\bibitem[Bau80]{Baurfourss}
W.~Baur.
\newblock On the elementary theory of quadruples of vector spaces.
\newblock {\em Ann. Math. Logic}, 19(3):243--262, 1980.

\bibitem[Bon81]{Bongartz}
K.~Bongartz.
\newblock Tilted algebras.
\newblock In {\em Representations of algebras ({P}uebla, 1980)}, volume 903 of
  {\em Lecture Notes in Math.}, pages 26--38. Springer, Berlin-New York, 1981.

\bibitem[BP02]{ZieZardomstring}
K.~Burke and M.~Prest.
\newblock The {Z}iegler and {Z}ariski spectra of some domestic string algebras.
\newblock {\em Algebr. Represent. Theory}, 5(3):211--234, 2002.

\bibitem[But80]{Conass}
M.~C.~R. Butler.
\newblock The construction of almost split sequences. {I}.
\newblock {\em Proc. London Math. Soc. (3)}, 40(1):72--86, 1980.

\bibitem[CF04]{ColbyFuller}
R.~R. Colby and K.~R. Fuller.
\newblock {\em Equivalence and duality for module categories}, volume 161 of
  {\em Cambridge Tracts in Mathematics}.
\newblock Cambridge University Press, Cambridge, 2004.
\newblock With tilting and cotilting for rings.

\bibitem[Dro79]{Drozd}
Ju.~A. Drozd.
\newblock Tame and wild matrix problems.
\newblock In {\em Representations and quadratic forms ({R}ussian)}, pages
  39--74, 154. Akad. Nauk Ukrain. SSR, Inst. Mat., Kiev, 1979.

\bibitem[EGN{\etalchar{+}}98]{HandRecMath}
Yu.~L. Ershov, S.~S. Goncharov, A.~Nerode, J.~B. Remmel, and V.~W. Marek,
  editors.
\newblock {\em Handbook of recursive mathematics. {V}ol. 1}, volume 138 of {\em
  Studies in Logic and the Foundations of Mathematics}.
\newblock North-Holland, Amsterdam, 1998.
\newblock Recursive model theory.

\bibitem[End01]{Enderton}
H.~B. Enderton.
\newblock {\em A mathematical introduction to logic}.
\newblock Harcourt/Academic Press, Burlington, MA, second edition, 2001.

\bibitem[Gei94]{Geislerthesis}
G.~Geisler.
\newblock {\em {Zur Modelltheorie von Moduln}}.
\newblock Phd thesis, {Universit\"at Freiburg}, 1994.

\bibitem[GL87]{GeiLen}
W.~Geigle and H.~Lenzing.
\newblock A class of weighted projective curves arising in representation
  theory of finite-dimensional algebras.
\newblock In {\em Singularities, representation of algebras, and vector bundles
  ({L}ambrecht, 1985)}, volume 1273 of {\em Lecture Notes in Math.}, pages
  265--297. Springer, Berlin, 1987.

\bibitem[GL91]{GeiLenperp}
W.~Geigle and H.~Lenzing.
\newblock Perpendicular categories with applications to representations and
  sheaves.
\newblock {\em J. Algebra}, 144(2):273--343, 1991.

\bibitem[GP16]{intandreptype}
L.~Gregory and M.~Prest.
\newblock Representation embeddings, interpretation functors and controlled
  wild algebras.
\newblock {\em J. Lond. Math. Soc. (2)}, 94(3):747--766, 2016.

\bibitem[Har11]{Richardthesis}
R.~Harland.
\newblock {\em Pure-injective Modules over Tubular Algebras and String
  Algebras}.
\newblock PhD thesis, University of Manchester, 2011.
\newblock available at www.maths.manchester.ac.uk/$\sim$
  mprest/publications.html.

\bibitem[Her93]{herzogduality}
I.~Herzog.
\newblock Elementary duality of modules.
\newblock {\em Trans. Amer. Math. Soc.}, 340(1):37--69, 1993.

\bibitem[HP15]{modirrslope}
R.~Harland and M.~Prest.
\newblock Modules with irrational slope over tubular algebras.
\newblock {\em Proc. Lond. Math. Soc. (3)}, 110(3):695--720, 2015.

\bibitem[Kra98]{KrauseGeneric}
H.~Krause.
\newblock Generic modules over {A}rtin algebras.
\newblock {\em Proc. London Math. Soc. (3)}, 76(2):276--306, 1998.

\bibitem[Kus97]{Kussinthesis}
D.~Kussin.
\newblock {\em Graduierte Faktorialit\"at und die Parameterkurven tubularer
  Familien}.
\newblock PhD thesis, Universit\"at-Gesamthochschule Paderborn, 1997.

\bibitem[Kus00]{KussinKtheory}
D.~Kussin.
\newblock On the {$K$}-theory of tubular algebras.
\newblock {\em Colloq. Math.}, 86(1):137--152, 2000.

\bibitem[LdlPn99]{LenPensep}
H.~Lenzing and J.~A. de~la Pe\~{n}a.
\newblock Concealed-canonical algebras and separating tubular families.
\newblock {\em Proc. London Math. Soc. (3)}, 78(3):513--540, 1999.

\bibitem[Len83]{LenzingHomtransfer}
H.~Lenzing.
\newblock Homological transfer from finitely presented to infinite modules.
\newblock In {\em Abelian group theory ({H}onolulu, {H}awaii, 1983)}, volume
  1006 of {\em Lecture Notes in Math.}, pages 734--761. Springer, Berlin, 1983.

\bibitem[LM92]{LenMeltub}
H.~Lenzing and H.~Meltzer.
\newblock Sheaves on a weighted projective line of genus one and
  representations of a tubular algebra.
\newblock In {\em Proceedings of the {S}ixth {I}nternational {C}onference on
  {R}epresentations of {A}lgebras ({O}ttawa, {ON}, 1992)}, volume~14 of {\em
  Carleton-Ottawa Math. Lecture Note Ser.}, page~25. Carleton Univ., Ottawa,
  ON, 1992.

\bibitem[LM96]{LenMeltilt}
H.~Lenzing and H.~Meltzer.
\newblock Tilting sheaves and concealed-canonical algebras.
\newblock In {\em Representation theory of algebras ({C}ocoyoc, 1994)},
  volume~18 of {\em CMS Conf. Proc.}, pages 455--473. Amer. Math. Soc.,
  Providence, RI, 1996.

\bibitem[Mar02]{Markermodeltheory}
D.~Marker.
\newblock {\em Model theory: An introduction}, volume 217 of {\em Graduate
  Texts in Mathematics}.
\newblock Springer-Verlag, New York, 2002.

\bibitem[Pre85]{decmiketameher}
M.~Prest.
\newblock Tame categories of modules and decidability.
\newblock {\em preprint}, 1985.

\bibitem[Pre88]{PreBk}
M.~Prest.
\newblock {\em Model theory and modules}, volume 130 of {\em London
  Mathematical Society Lecture Note Series}.
\newblock Cambridge University Press, Cambridge, 1988.

\bibitem[Pre96]{Epiintandstrwild}
M.~Prest.
\newblock Epimorphisms of rings, interpretations of modules and strictly wild
  algebras.
\newblock {\em Comm. Algebra}, 24(2):517--531, 1996.

\bibitem[Pre97]{PreInterp}
M.~Prest.
\newblock Interpreting modules in modules.
\newblock {\em Ann. Pure Appl. Logic}, 88(2-3):193--215, 1997.
\newblock Joint AILA-KGS Model Theory Meeting (Florence, 1995).

\bibitem[Pre09]{PSL}
M.~Prest.
\newblock {\em Purity, spectra and localisation}, volume 121 of {\em
  Encyclopedia of Mathematics and its Applications}.
\newblock Cambridge University Press, Cambridge, 2009.

\bibitem[Pre11]{defaddcats}
M.~Prest.
\newblock Definable additive categories: purity and model theory.
\newblock {\em Mem. Amer. Math. Soc.}, 210(987):vi+109, 2011.

\bibitem[PT09]{decfinrings}
G.~Puninski and C.~Toffalori.
\newblock Towards the decidability of the theory of modules over finite
  commutative rings.
\newblock {\em Ann. Pure Appl. Logic}, 159(1-2):49--70, 2009.

\bibitem[Rin84]{Ringeltub}
C.~M. Ringel.
\newblock {\em Tame algebras and integral quadratic forms}, volume 1099 of {\em
  Lecture Notes in Mathematics}.
\newblock Springer-Verlag, Berlin, 1984.

\bibitem[Rin98]{ZgspcetameherRingel}
C.~M. Ringel.
\newblock The {Z}iegler spectrum of a tame hereditary algebra.
\newblock {\em Colloq. Math.}, 76(1):105--115, 1998.

\bibitem[RR06]{InfdimcanalgReitenRingel}
I.~Reiten and C.~M. Ringel.
\newblock Infinite dimensional representations of canonical algebras.
\newblock {\em Canad. J. Math.}, 58(1):180--224, 2006.

\bibitem[SF75]{SlobFrid}
A.~M. Slobodsko{\u\i} and {\`E}.~I. Fridman.
\newblock Theories of abelian groups with predicates that distinguish
  subgroups.
\newblock {\em Algebra i Logika}, 14(5):572--575, 607, 1975.

\bibitem[Sko90]{Skopolgrowth}
A.~Skowro\'{n}ski.
\newblock Algebras of polynomial growth.
\newblock In {\em Topics in algebra, {P}art 1 ({W}arsaw, 1988)}, volume~26 of
  {\em Banach Center Publ.}, pages 535--568. PWN, Warsaw, 1990.

\bibitem[SS07]{ss3}
D.~Simson and A.~Skowro{\'n}ski.
\newblock {\em Elements of the representation theory of associative algebras.
  {V}ol. 3}, volume~72 of {\em London Mathematical Society Student Texts}.
\newblock Cambridge University Press, Cambridge, 2007.
\newblock Representation-infinite tilted algebras.

\end{thebibliography}
\end{document}